\documentclass{jms_My_en}
\paperTitle{The calculation of the distribution function of a strictly stable law at large X}
\articleColonName{The calculation of a strictly stable law at $x\to\infty$}
\authorsShort{V.\,V.\,Saenko}
\authorsFull{V.\,V. Saenko\first}
\addAuthorInfo{Ulyanovsk State University, S.P. Kapitsa Research Institute of Technology, 42, Leo Tolstoy St., Ulyanovsk, 432017, email address: \url{vvsaenko@inbox.ru}}

\paperAbstract{The paper considers the problem of calculating the distribution function of a strictly stable law at $x\to\infty$. To solve this problem, an expansion of the distribution function in a power series was obtained, and an estimate of the remainder term was also obtained. It was shown that in the case $\alpha<1$ this series was convergent for any $x$, in the case $\alpha=1$ the series was convergent at $N\to\infty$ in the domain $|x|>1$, and in the case $\alpha>1$ the series was asymptotic at $x\to\infty$. The case $\alpha=1$ was considered separately and it was demonstrated that in that case the series converges to the generalized Cauchy distribution. An estimate for the threshold coordinate $x_\varepsilon^N$ was obtained which determined the area of applicability of the obtained expansion. It was shown that in the domain $|x|\geqslant x_\varepsilon^N$ this power series could be used to calculate the distribution function, which completely solved the problem of calculating the distribution function at large $x$.}

\usepackage{color}
\begin{document}
\maketitle

\section{Introduction}

The main method for calculating the probability density and the distribution function of stable laws is the use of integral representations of these quantities. The reason for this situation is the impossibility of obtaining expressions for these quantities in elementary functions in the general case. The exception comprises only five cases: the Levy distribution ($\alpha=1/2,\theta=1$), the symmetric Levy distribution ($\alpha=1/2,\theta=-1$), the Cauchy distribution ($\alpha=1,\theta=0$), Gaussian distribution ($\alpha=2,\theta=0$) and generalized Cauchy distribution ($\alpha=1,-1\leqslant\theta\leqslant1$).

When performing the inverse Fourier transform of the characteristic function, it is possible to obtain two types of integral representations.  The first type includes representations expressing the probability density and the distribution function in terms of an improper integral of the oscillating function. The works \cite{Nolan1999,Ament2018} are devoted to  obtaining and studying such representations. The second type includes integral representations expressing the probability density and distribution function in terms of a definite integral of a monotone function.  The works\cite{Zolotarev1964_en,Zolotarev1986,Uchaikin1999} are devoted to obtaining and studying integral representations in the parameterization ``B",  the works  \cite{Nolan1997, Nolan2022}, are devoted to obtaining and studying integral representations in the parameterization  ``M"  and the paper \cite{Saenko2020b} to devoted to the parameterization ``C". Here, to determine various parametrizations of the characteristic function of the stable law, the notation was used which was introduced in the book  \cite{Zolotarev1986}. Further in the text, we will continue adhering to these notations.

Integral representations of the second type are most widely used due to the convenience of their use. The method of the inverse Fourier transform, which leads to this type of integral representations, is called the stationary phase method. The convenience of using such representations lies in the fact that the integrand is a monotonic function and in a wide range of coordinates and parameters there are no difficulties in calculating the definite integral of such a function. The integral representations for the parameterization ``M" of the characteristic function gained in popularity. These integral representations served as the basis for the development of several software products \cite{Liang2013,Royuela-del-Val2017,Julian-Moreno2017,Rimmer2005,Veillette2008}.

Both the first and the second type of the integral representation of a stable law have their disadvantages. The main difficulty in using the integral representation of the first type is the oscillating integrand. In some cases, numerical methods are cannot calculate the integral of such a function. In particular,  the work \cite{Nolan1999} indicates the following problems for the integral representation in the parameterization ``M": 1)~in the case $\alpha<0.75$ the integration domain becomes very large, which leads to difficulties in numerical integration; the integration domain becomes very large which leads to difficulties in numerical integration; 2)~if $\beta\neq0$ and $0<|\alpha-1|<0.001$ there are calculation problems in calculating the term with $(\tan(\pi\alpha/2)(t-t^\alpha)$; 3)~when $x$ is very large, the integrand oscillates very quickly. In the paper \cite{Ament2018} the authors propose to modernize the standard quadrature numerical integration algorithm to adapt the calculation of integrals of an oscillating function. This gave an opportunity to reduce the lower limit of the parameter $\alpha$ from 0.75 to 0.5. To calculate the probability density at large $x$ it is proposed to use the expansion of the probability density in a power series. However, the paper points out that the proposed scheme is not applicable for symmetric distributions in the case of $\alpha<0.5$ and for asymmetric distributions in the cases of $\alpha<0.5$ and $0.9<\alpha<1.1$.

The second type of integral representations also has some features that lead to difficulties in numerical integration.  The cause of the calculation difficulties is the behavior of the integrand at very small and very large values of the coordinate $x$. In the case of the integral representation for the probability density in these two cases, the integrand has the form of a very narrow peak. As a result, numerical integration algorithms cannot determine this peak and give an incorrect integration result. This behavior of the integrand is pointed out in the articles \cite{Nolan1997,Pogany2015,Royuela-del-Val2017,Ament2018}. To eliminate this problem various numerical algorithms are proposed to use in the papers \cite{Nolan1997,Royuela-del-Val2017,Julian-Moreno2017} However, all these algorithms increase the accuracy of calculations, but do not eliminate the problem completely.

To solve this problem, it is expedient to use the approaches not possessing any specific features in these areas that can lead to calculation difficulties. The most suitable idea is to use expansions of the probability density and distribution function in power series at $x\to0$ and $x\to\infty$. The articles \cite{Saenko2022b,Saenko2023} show that  the use of expansions of stable laws in power series at  $x\to0$ and $x\to\infty$ makes it possible to solve the problem of calculating stable laws completely at very small and very large $x$. However, in these articles the problem of calculating the distribution function of a strictly stable law in the case of $x\to\infty$ was left out of consideration. Therefore, the main purpose of this paper is to fill this gap.

This paper considers the problem of calculating the distribution function in the case of $x\to\infty$ with the characteristic function
\begin{equation}\label{eq:CF_formC}
  \hat{g}(t,\alpha,\theta,\lambda)=\exp\left\{-\lambda |t|^\alpha\exp\{-i\tfrac{\pi}{2}\alpha\theta\sign t\}\right\},\quad t\in\mathbf{R},
\end{equation}
where $\alpha\in(0,2]$, $|\theta|\leqslant\min(1,2/\alpha-1)$, $\lambda>0$. According to the terminology introduced in the book \cite{Zolotarev1986}, this characteristic function corresponds to the parameterization ``C". In the paper \cite{Saenko2020b} the inverse Fourier transform of this characteristic function was performed using the stationary phase method, and integral representations for the probability density and distribution function were obtained (see~Appendix~\ref{sec:IntRepr}). As one can see the formula (\ref{eq:G(x)_int}) expresses the distribution function in terms of a definite integral, and belongs to the second type of integral representations. In the general case, the integrand in (\ref{eq:G(x)_x>0}) is a monotonic function varying from 0 to 1 as the integration variable changes from the lower limit point $\varphi=-\pi\theta/2$ to the upper limit point $\varphi=\pi/2$. In the case $0<\alpha<1$ it is a decreasing function, and in the case of $1<\alpha\leqslant2$ it is an increasing function. However, at very small and very large values of $x$ the change in the function from 0 to 1 occurs so fast that numerical integration algorithms cannot recognize it. As a result, this leads to an incorrect calculation of the integral and points to the fact that in this range of coordinates it is no longer possible to use the integral representation (\ref{eq:G(x)_int}) to calculate the distribution function. In this paper, to calculate the distribution function in the specified range of coordinates, it is proposed to use the expansion of the distribution function in a power series at $x\to\infty$. To do this, such an expansion of the distribution function will be obtained and the conditions for the applicability of this expansion will be determined.

\section{Representation of the distribution function as a power series}

We will obtain the expansion of the distribution function in a series at $x\to\infty$ for a strictly stable law with a characteristic function (\ref{eq:CF_formC}). Without loss of generality, we will assume that the scale parameter is $\lambda=1$. It is generally accepted to call strictly stable laws with the scale parameter $\lambda=1$ standard strictly stable laws and shorthand notations are used for them.  The characteristic function is usually denoted as $\hat{g}(t,\alpha,\theta,1)\equiv\hat{g}(t,\alpha,\theta)$, the probability density - $g(x,\alpha,\theta,1)\equiv g(x,\alpha,\theta)$, the distribution function - $G(x,\alpha,\theta,1)\equiv G(x,\alpha,\theta)$,  a strictly stable random quanitity $Y(\alpha,\theta,1)\equiv Y(\alpha,\theta)$. Further in the text we will use this notation.  It should be noted that to transform a standard strictly stable law into a strictly stable law with an arbitrary $\lambda$ one can use remark 5 and remark 7 from the paper \cite{Saenko2020b}, (see also\cite{Zolotarev1986, Zolotarev1999}).

 We also need the inversion property, which for a standard strictly stable law with the characteristic function (\ref{eq:CF_formC}) has the form

\begin{property}\label{prop:Inversion}
For any admissible parameters $(\alpha,\theta)$
\begin{equation*}
  Y(\alpha,-\theta)\stackrel{d}{=}-Y(\alpha,\theta).
\end{equation*}
\end{property}

The proof of this property was given in the paper \cite{Saenko2020b} (see also \cite{Zolotarev1986,Uchaikin1999}).   In the terms of the distribution function $G(x,\alpha,\theta)$  this property takes the form
\begin{equation}\label{eq:InversionFormula}
  G(-x,\alpha,\theta)=1-G(x,\alpha,-\theta).
\end{equation}
The convenience of this property lies in the fact that, when studying the distribution function, it gives us an opportunity to confine ourselves to considering only the case $x\geqslant0$. Expressions for the case $x<0$ are obtained using this formula.

To solve the stated problem, we need to expand the probability density into a series at $x\to\infty$. A similar expansion was obtained in the article \cite{Saenko2023}, where the following theorem was proved.

\begin{theorem}\label{theor:pdfExpansion}
In the case $x\to\pm\infty$ for any admissible set of parameters $(\alpha,\theta)$ except for the values $\theta=\pm1$ for the probability density $g(x,\alpha,\theta)$ the representation in the form of a power series is valid
\begin{equation*}
  g(x,\alpha,\theta)=g_N^{\infty}(|x|,\alpha,\theta^*)+R_N^{\infty}(|x|,\alpha,\theta^*),
\end{equation*}
where $\theta^*=\theta\sign(x)$ and
\begin{equation}\label{eq:g_NInf}
  g_N^\infty(x,\alpha,\theta)=\frac{1}{\pi}\sum_{n=0}^{N-1}\frac{(-1)^{n+1}}{n!}\Gamma(\alpha n+1)\sin\left(\tfrac{\pi}{2}\alpha n(1+\theta)\right) x^{-\alpha n-1},\quad x>0,
\end{equation}
\begin{equation*}
|R_N^{\infty}(x,\alpha,\theta)|\leqslant
  \frac{x^{-\alpha N-1}}{\pi N!}\left(\Gamma(\alpha N+1)+x^{-\alpha}\Gamma(\alpha(N+1)+1)\right),\quad x>0.
\end{equation*}
\end{theorem}

Using this theorem, one can obtain an expansion of the distribution function at $x\to\infty$. As a result, the following theorem turns out to be true.

\begin{theorem}\label{theor:cdfExpansion}
   For any admissible values of parameters $(\alpha,\theta)$ except for the values $\theta=\pm1$ at $x\to\pm\infty$ for the distribution function $G(x,\alpha,\theta)$ the representation in the form of a power series is valid
  \begin{equation}\label{eq:G(x)}
    G(x,\alpha,\theta)=\tfrac{1}{2}(1+\sign(x))-\sign(x)\left(G_N^{\infty}(|x|,\alpha,\theta^*)+\mathcal{R}_{N}^{\infty}(|x|,\alpha,\theta^*)\right),
  \end{equation}
  where $\theta^*=\theta\sign(x)$,
  \begin{align}
    G_N^{\infty}(x,\alpha,\theta)&=\frac{1}{\pi}\sum_{n=1}^{N-1}\frac{(-1)^{n+1}}{n!}\Gamma(\alpha n)\sin\left(\tfrac{\pi}{2}\alpha n(1+\theta)\right)x^{-\alpha n},\quad x>0,\label{eq:G_NInf}\\
    |\mathcal{R}_N^{\infty}(x,\alpha,\theta)|&\leqslant\frac{x^{-\alpha N}}{\pi N!}\left(\Gamma(\alpha N)+x^{-\alpha}\Gamma(\alpha(N+1))\right),\quad x>0.\label{eq:RG_NInf}
  \end{align}
\end{theorem}

\begin{proof}
To prove it we will use theorem~\ref{theor:pdfExpansion}. Without loss of generality, we will consider the case $x>0$. The case $x<0$ can be obtained using the inversion property (\ref{eq:InversionFormula}). By definition, at $x>0$ the distribution function has the form
\begin{equation}\label{eq:G(+)}
  G^{(+)}(x,\alpha,\theta)=1-\int_{x}^{\infty} g(\xi,\alpha,\theta)d\xi,\quad x>0,
\end{equation}
where $g(x,\alpha,\theta)$  is the probability density of a strictly stable law and the superscript ``$(+)$" shows that this expression determines the distribution function on the positive part of the semiaxis.

Thus, the expansion of the distribution function in a series is determined by the expansion of the probability density in a series. It is known that the expansion of any function in a Taylor series consists of the $N$-th partial sum and the remainder term. Consequently, the expansion of the probability density $g(x,\alpha,\theta)$ can be written in the form

\begin{equation}\label{eq:g_expan_x>0}
    g(x,\alpha,\theta)=g_N^{\infty}(x,\alpha,\theta)+R_N^{\infty}(x,\alpha,\theta),\quad x>0,
  \end{equation}
where $g_N^\infty(x,\alpha,\theta)$ is the $N$-th partial sum and $R_N^\infty(x,\alpha,\theta)$ is the remainder term of the series. In the case $x\to\infty$ the first summand is determined by the expression (\ref{eq:g_NInf}) and for the remainder term, we use the expression obtained in the article \cite{Saenko2023}
  \begin{equation}\label{eq:R_NInf_0}
    R_N^\infty(x,\alpha,\theta) = \frac{1}{\pi x}\Re ie^{-i\frac{\pi}{2}\theta}\int_{0}^{\infty} \exp\left\{-\tau e^{-i\frac{\pi}{2}\theta}\right\} R_N\left(-\left(\tfrac{i\tau}{x}\right)^\alpha\right) d\tau, \quad x>0,
  \end{equation}
where $R_N(y)=\frac{y^N}{N!}e^{y\zeta}$, $(0<\zeta<1)$  is the remainder term in the Lagrange form.

It should be noted that in the case of $x\to\infty$ the expression (\ref{eq:g_expan_x>0}), as well as theorem~\ref{theor:pdfExpansion} and the expression~(\ref{eq:R_NInf_0}) are valid when the condition  $\tau/x\to0$ is satisfied, where $\tau$ is the integration variable in the inverse Fourier transform formula. In particular, the integration variable in the formula (\ref{eq:R_NInf_0}). See the paper \cite{Saenko2023} for detail. Hence, here, and further in the text, we will assume everywhere that $\tau/x\to0$.

To obtain the expansion of the distribution function in a power series at $x\to\infty$ we will substitute the expression (\ref{eq:g_expan_x>0}) in the expression (\ref{eq:G(+)}). As a result, we get
\begin{equation}\label{eq:G_tmp1}
  G^{(+)}(x,\alpha,\theta)=1-\int_{x}^{\infty}g_N^{\infty}(\xi,\alpha,\theta)d\xi-\int_{x}^{\infty}R_N^\infty(x,\alpha,\theta)d\xi,\quad x>0,
\end{equation}
where $g_N^\infty(x,\alpha,\theta)$ has the form (\ref{eq:g_NInf}), and $R(x,\alpha,\theta)$ is determined by the expression (\ref{eq:R_NInf_0}).

Interchanging the order of integration and summation in the second summand, we obtain
\begin{multline*}
  G_N^\infty(x,\alpha,\theta)=\int_{x}^{\infty}g_N^{\infty}(\xi,\alpha,\theta)d\xi
  =\frac{1}{\pi}\sum_{n=0}^{N-1}\frac{(-1)^{n+1}}{n!}\Gamma(\alpha n+1)\sin\left(\tfrac{\pi}{2}\alpha n(1+\theta)\right) \int_{x}^{\infty}\xi^{-\alpha n-1}d\xi=\\
  =\frac{1}{\pi}\sum_{n=0}^{N-1}\frac{(-1)^{n+1}}{n!}\Gamma(\alpha n)\sin\left(\tfrac{\pi}{2}\alpha n(1+\theta)\right) x^{-\alpha n}, \quad x>0.
\end{multline*}

We should pay attention that at $n=0$ the correspondent summand in the sum is equal to zero. Therefore, the summation can be started with $n=1$. As a result, we come to the expression (\ref{eq:G_NInf}).

Now we will obtain an expression for the remainder term.  Using the expression (\ref{eq:R_NInf_0}) and changing the order of integration in some places, for the third summand in (\ref{eq:G_tmp1}) we obtain
\begin{multline}\label{eq:R_NInf_tmp0}
  \mathcal{R}_N^{\infty}(x,\alpha,\theta)=\int_{x}^{\infty}R_N^{\infty}(\xi,\alpha,\theta)d\xi=
  \frac{1}{\pi}\Re i e^{-\frac{\pi}{2}\theta}\int_{x}^{\infty}\frac{d\xi}{\xi}\int_{0}^{\infty}\exp\left\{-\tau e^{-\frac{\pi}{2}\theta}\right\} R_N\left(-\left(\frac{i\tau}{\xi}\right)^\alpha\right)d\tau=\\
  = \frac{1}{\pi}\Re i e^{-\frac{\pi}{2}\theta}\int_{x}^{\infty}\frac{d\xi}{\xi}\int_{0}^{\infty} \exp\left\{-\tau e^{-i\frac{\pi}{2}\theta}\right\}\frac{1}{N!}\left(-\left(\frac{i\tau}{\xi}\right)^\alpha\right)^N \exp\left\{-\left(\frac{i\tau}{\xi}\right)^\alpha\zeta\right\}d\xi=\\
  =\frac{1}{\pi N!}\Re i e^{-\frac{\pi}{2}\theta}\int_{0}^{\infty}\exp\left\{-\tau e^{-i\frac{\pi}{2}\theta}\right\} (-(i\tau)^\alpha)^N d\tau \int_{x}^{\infty}\xi^{-\alpha N-1}\exp\left\{-\left(\frac{i\tau}{\xi}\right)^\alpha\zeta\right\}d\xi
\end{multline}

Unfortunately, we cannot calculate the integral since the exact value of the variable $\zeta$ is unknown. It is only known that this variable takes values from the interval $0<\zeta<1$. Nevertheless, it is possible to estimate the value of this integral. We will consider $|\mathcal{R}_N^{\infty}(x,\alpha,\theta)|$. Taking into account that the case $\tau/x\to0$ is being considered, we can expand the multiplier $\exp\left\{-\left(\tfrac{i\tau}{\xi}\right)^\alpha\zeta\right\}$ in a Taylor series and leave only the summands of the first order of smallness. We have
\begin{equation*}
 \exp\left\{-\left(\frac{i\tau}{x}\right)^\alpha\zeta\right\} =\sum_{k=0}^{\infty}\frac{(-1)^k}{k!}\left(\left(\frac{i\tau}{x}\right)^\alpha\zeta\right)^k\approx1-\zeta\left(\frac{i\tau}{x}\right)^\alpha
\end{equation*}
To calculate the obtained integral we will also need the following formula given in \cite{Bateman_V1_1953} (see \S 1.5, the formula (31))
  \begin{equation*}
   \int_{0}^{\infty}t^{\gamma-1}e^{-ct\cos\beta-ict\sin\beta}dt=\Gamma(\gamma)c^{-\gamma}e^{-i\gamma\beta},\
   -\frac{\pi}{2}<\beta<\frac{\pi}{2},\ \Re\gamma>0\ \text{or}\ \beta=\pm\frac{\pi}{2},\ 0<\Re\gamma<1.
  \end{equation*}
If we use the Euler formula $\cos\beta+i\sin\beta=e^{i\beta}$, then this integral can be represented in the form
  \begin{equation}\label{eq:Gamma_intRepr1}
   \int_{0}^{\infty}t^{\gamma-1}e^{-ct\exp\{i\beta\}}dt=\Gamma(\gamma)c^{-\gamma}e^{-i\gamma\beta},\quad
   -\frac{\pi}{2}<\beta<\frac{\pi}{2},\ \Re\gamma>0\ \text{or}\ \beta=\pm\frac{\pi}{2},\ 0<\Re\gamma<1.
  \end{equation}
Taking into account that $x>0$ for (\ref{eq:R_NInf_tmp0}) the following estimates turn out to be valid
\begin{multline*}
  |\mathcal{R}_N^{\infty}(x,\alpha,\theta)| =\frac{1}{\pi N!}\left|\Re i e^{-\frac{\pi}{2}\theta}\int_{0}^{\infty}\exp\left\{-\tau e^{-i\frac{\pi}{2}\theta}\right\} (-(i\tau)^\alpha)^N d\tau \int_{x}^{\infty}\xi^{-\alpha N-1}\exp\left\{-\left(\frac{i\tau}{\xi}\right)^\alpha\zeta\right\}d\xi\right|\\
  \leqslant \frac{1}{\pi N!}\left| \int_{0}^{\infty}\exp\left\{-\tau e^{-i\frac{\pi}{2}\theta}\right\} (-(i\tau)^\alpha)^N d\tau \int_{x}^{\infty} \xi^{-\alpha N-1}\left(1-(i\tau)^\alpha\xi^{-\alpha}\zeta\right)d\xi\right|  =\\
  = \frac{1}{\pi N!}\left| \int_{0}^{\infty}\exp\left\{-\tau e^{-i\frac{\pi}{2}\theta}\right\} (-(i\tau)^\alpha)^N \left(\frac{x^{-\alpha N}}{\alpha N} -(i\tau)^\alpha\frac{\zeta x^{-\alpha(N+1)}}{\alpha(N+1)} \right) d\tau \right|\leqslant\\
  \leqslant\frac{1}{\pi N!}\frac{x^{-\alpha N}}{\alpha N}\left|\int_{0}^{\infty}\exp\left\{-\tau e^{-i\frac{\pi}{2}\theta}\right\}\tau^{\alpha N}d\tau\right|+
  \frac{1}{\pi N!}\frac{\zeta x^{-\alpha(N+1)}}{\alpha(N+1)} \left|\int_{0}^{\infty}\exp\left\{-\tau e^{-i\frac{\pi}{2}\theta}\right\}\tau^{\alpha (N+1)}d\tau\right|=\\
  = \frac{x^{-\alpha N}}{\pi N!}\left(\frac{\Gamma(\alpha N+1)}{\alpha N}+\zeta x^{-\alpha}\frac{\Gamma(\alpha(N+1)+1)}{\alpha(N+1)}\right)\leqslant
  \frac{x^{-\alpha N}}{\pi N!}\left(\Gamma(\alpha N)+x^{-\alpha}\Gamma(\alpha(N+1))\right).
\end{multline*}
Here, when passing in the last equality, the formula (\ref{eq:Gamma_intRepr1}) was used to calculate the integrals, and when passing in the last inequality, it was assumed that $\zeta=1$.

To substantiate the validity of using the formula (\ref{eq:Gamma_intRepr1}) when calculating the integrals in this expression, we examine the range of the argument $-\tfrac{\pi}{2}\theta$. The range of admissible values for the parameter $\theta$ is determined by the inequality $|\theta|\leqslant\min(1,2/\alpha-1)$. From this it follows that if $\alpha\leqslant1$, then $-1\leqslant\theta\leqslant1$, and if $1<\alpha\leqslant2$, then $-(2/\alpha-1)\leqslant\theta\leqslant2/\alpha-1$. Thus, for any $0<\alpha\leqslant2$ we obtain  $ -\tfrac{\pi}{2}\leqslant-\tfrac{\pi}{2}\theta\leqslant\tfrac{\pi}{2}$. The extreme values of this interval $\pm\tfrac{\pi}{2}$ are attained at the values $\alpha\leqslant1$ and $\theta=\mp1$. Now we will compare the integral in (\ref{eq:Gamma_intRepr1}) with the integrals in the expression during the passage in the last equality. We see that the integral (\ref{eq:Gamma_intRepr1})  coincides with these integrals except for the case $-\tfrac{\pi}{2}\theta=\pm\tfrac{\pi}{2}$. These two points are out of the range of admissible values  for the argument $\beta$ in the formula (\ref{eq:Gamma_intRepr1}). Therefore, they should be excluded from consideration.

Now getting back to (\ref{eq:G_tmp1}), we obtain
\begin{equation}\label{eq:G(+)_1}
  G^{(+)}(x,\alpha,\theta)=1-G_N^{\infty}(x,\alpha,\theta)-\mathcal{R}_N^\infty(x,\alpha,\theta),\quad x>0,
\end{equation}
where for $\mathcal{R}_N^\infty(x,\alpha,\theta)$ the estimate is valid
\begin{equation*}
  \left|\mathcal{R}_N^\infty(x,\alpha,\theta)\right|\leqslant\frac{x^{-\alpha N}}{\pi N!}\left(\Gamma(\alpha N)+x^{-\alpha}\Gamma(\alpha(N+1))\right),\quad x>0.
\end{equation*}

Since the case $x>0$ has been considered so far, these expressions are valid at $x>0$. To obtain the expansion of the distribution function at $x<0$ we will use the inversion property. Using in the formula (\ref{eq:InversionFormula}) the expression (\ref{eq:G(+)_1}) we obtain
\begin{equation*}
  G^{(-)}(-x,\alpha,\theta)=G_N^{\infty}(x,\alpha,-\theta)+\mathcal{R}_N^\infty(x,\alpha,-\theta),\quad x>0.
\end{equation*}
If we now introduce the notation $\theta^*=\theta\sign(x)$ and take the coordinate $x$ in absolute value, then we can combine the formulas for $G^{(+)}(x,\alpha,\theta)$ and $G^{(-)}(x,\alpha,\theta)$ into one formula. As a result, we obtain the expression (\ref{eq:G(x)}). Thus, the theorem is proved.
\begin{flushright}
  $\Box$
\end{flushright}
\end{proof}

The proved theorem determines the expansion of the distribution function of a strictly stable law with characteristic function (\ref{eq:CF_formC}) into a power series at $x\to\infty$. Now we examine the issue of the convergence of the obtained expansion. Since this expansion was obtained by integrating the expansion for the probability density, taking into account the results of Corollary 1, proved in the paper \cite{Saenko2023},  one can state that this series converges in the case of $\alpha<1$ for all $x$, in the case $\alpha=1$, only for $|x|>1$, and in the case $\alpha>1$ the series is asymptotic one at  $x\to\infty$. A more precise formulation is given by the following corollary.

\begin{corollary}\label{corol:cdfConverg}
  In the case $\alpha<1$ the series (\ref{eq:G_NInf}) converges for any $x$ at $N\to\infty$. In this case for the distribution function $G(x,\alpha,\theta)$  for any admissible $\theta$ the representation is valid in the form of an infinite series
      \begin{equation}\label{eq:G_a<1}
        G(x,\alpha,\theta)=\tfrac{1}{2}(1+\sign(x))-\frac{\sign(x)}{\pi}\sum_{n=1}^{\infty}\frac{(-1)^{n+1}}{n!}\Gamma(\alpha n) \sin\left(\tfrac{\pi}{2}\alpha n(1+\theta^*)\right)|x|^{-\alpha n}.
      \end{equation}
   In the case $\alpha=1$ the series (\ref{eq:G_NInf}) converges for the values $|x|>1$ at $N\to\infty$. In this case the representation is valid in the form of an infinite series for the distribution function $G(x,1,\theta)$ for any admissible $\theta$
      \begin{equation}\label{eq:G_a=1}
        G(x,1,\theta)=\tfrac{1}{2}(1+\sign(x))-\frac{\sign(x)}{\pi}\sum_{n=1}^{\infty}\frac{(-1)^{n+1}}{n} \sin\left(\tfrac{\pi}{2}n(1+\theta^*)\right)|x|^{-n},\quad |x|>1.
      \end{equation}
  In the case $\alpha>1$ the series (\ref{eq:G_NInf}) diverges for any $x$ at $N\to\infty$. In this case the asymptotic expansion is valid for the distribution function $G(x,\alpha,\theta)$ for any admissible $\theta$
      \begin{equation}\label{eq:G_a>1}
        G(x,\alpha,\theta)\sim\tfrac{1}{2}(1+\sign(x))-\frac{\sign(x)}{\pi}\sum_{n=1}^{N-1}\frac{(-1)^{n+1}}{n!}\Gamma(\alpha n) \sin\left(\tfrac{\pi}{2}\alpha n(1+\theta^*)\right)|x|^{-\alpha n},\quad x\to\pm\infty.
      \end{equation}
Here, everywhere $\theta^*=\theta\sign(x)$.
\end{corollary}

\begin{proof}
Without loss of generality, we first consider the case $x>0$. The expansion for the case $x<0$ will be obtained using the inversion property (\ref{eq:InversionFormula}). It was previously obtained that at positive $x$ the representation (\ref{eq:G(+)_1}) is valid. From this expression and also from (\ref{eq:RG_NInf}) it follows that
\begin{equation}\label{eq:G(+)-GN}
  |G^{(+)}(x,\alpha,\theta)-1+G_N^\infty(x,\alpha,\theta)|\leqslant \frac{x^{-\alpha N}}{\pi N!}\left(\Gamma(\alpha N)+x^{-\alpha}\Gamma(\alpha(N+1))\right),\quad x>0.
\end{equation}
We examine the convergence of the series (\ref{eq:G_NInf}). Since this series is sign-alternating, the inequalities are valid
\begin{multline*}
  G_N^\infty(x,\alpha,\theta)\leqslant|G_N^\infty(x,\alpha,\theta)|\leqslant
  \frac{1}{\pi} \sum_{n=1}^{N-1}\left|\frac{(-1)^{n+1}}{n!}\Gamma(\alpha n)\sin\left(\tfrac{\pi}{2}\alpha n(1+\theta)\right) x^{-\alpha n}\right| \\
  \leqslant\frac{1}{\pi}\sum_{n=1}^{N-1}\frac{\Gamma(\alpha n)}{\Gamma(n+1)}x^{-\alpha n},\quad x>0.
\end{multline*}
We will make use of the Cauchy criterion in the limiting form and the Stirling formula
\begin{equation}\label{eq:Stirling}
      \Gamma(z)\sim e^{-z}z^{z-\frac{1}{2}}\sqrt{2\pi},\quad z\to\infty, |\arg z|<\pi.
\end{equation}
As a result, we obtain
\begin{multline}\label{eq:SeriesTermLim}
  \lim_{n\to\infty}\left(\frac{1}{\pi}\frac{\Gamma(\alpha n)}{\Gamma(n+1)}x^{-\alpha n}\right)^{1/n}
  =\lim_{n\to\infty}\left(\frac{e^{-\alpha n} (\alpha n)^{\alpha n-1/2}\sqrt{2\pi} x^{-\alpha n}}{\pi e^{-n-1}(n+1)^{n+1-1/2}\sqrt{2\pi}}\right)^{1/n}
  =\lim_{n\to\infty}\frac{e^{-\alpha}(\alpha n)^{\alpha-\frac{1}{2n}}x^{-\alpha}}{\pi^{\frac{1}{n}}e^{-1-\frac{1}{n}} (n+1)^{1+\frac{1}{2n}}}\\
  = e^{1-\alpha}\alpha^\alpha x^{-\alpha} \lim_{n\to\infty}n^\alpha (n+1)^{-1}= e^{1-\alpha}\alpha^\alpha x^{-\alpha} \lim_{n\to\infty}n^{\alpha-1}=\begin{cases}
                                  0, & \mbox{if}\ \alpha<1, \\
                                  x^{-1}, & \mbox{if}\ \alpha=1, \\
                                  \infty, & \mbox{if}\ \alpha>1.
                                \end{cases}
\end{multline}
From this it is clear that in the case $\alpha<1$ the series (\ref{eq:G_NInf}) is convergent for any $x$, in the case $\alpha=1$ this series converges at $x>1$ and diverges at $x\leqslant1$. In the case $\alpha>1$ this series diverges for any $x$.

We examine the behavior of the remainder term (\ref{eq:RG_NInf}) in the case $N\to\infty$. Using the Stirling’s formula
 (\ref{eq:Stirling}) and taking into account that $N+1\approx N$ at $N\to\infty$ we obtain
\begin{multline}\label{eq:RN_lim}
 \lim_{N\to\infty} \mathcal{R}_N^\infty(x,\alpha,\theta)\leqslant \lim_{N\to\infty} \frac{x^{-\alpha N}}{\pi N!}\left(\Gamma(\alpha N)+x^{-\alpha}\Gamma(\alpha(N+1))\right)\\
  =\frac{1}{\pi}\lim_{N\to\infty} \frac{x^{-\alpha N}\Gamma(\alpha N)+x^{-\alpha(N+1)}\Gamma(\alpha(N+1))}{\Gamma(N+1)}
  \approx \frac{2}{\pi}\lim_{N\to\infty} x^{-\alpha N}\frac{\Gamma(\alpha N)}{\Gamma(N)}\\
  = \frac{2}{\pi}\lim_{N\to\infty}x^{-\alpha N}\frac{e^{-\alpha N} (\alpha N)^{\alpha N-1/2} \sqrt{2\pi}} {e^{-N}N^{N-1/2}\sqrt{2\pi}}
  =\frac{2}{\pi}\lim_{N\to\infty} x^{-\alpha N} e^{N(1-\alpha)}\alpha^{\alpha N-1/2} N^{N(\alpha-1)}\\
  =\frac{2}{\pi\sqrt{\alpha}}\lim_{N\to\infty}\exp\left\{N(1-\alpha)(1-\ln N)+\alpha N(\ln\alpha-\ln x)\right\}
  =\begin{cases}
     0, & \mbox{if }\alpha<1 \\
     \infty, & \mbox{if } \alpha=1, x\leqslant1,\\
     0, &\mbox{if }\alpha=1, x>1,\\
     \infty, & \mbox{if } \alpha>1.
   \end{cases}
\end{multline}

We will consider the case $\alpha<1$. Generalizing the results obtained, we see that in this case the series (\ref{eq:G_NInf}) is convergent, and the limit of the remainder term $\mathcal{R}_N^\infty(x,\alpha,\theta)$ is equal to zero. From this it follows that the right part  (\ref{eq:G(+)-GN}) is an element of an infinitesimal sequence. In turn, this means that for any fixed $x$ the sequence $1-G_N^\infty(x,\alpha,\theta)$ converges to the distribution function $G^{(+)}(x,\alpha,\theta)$ at $N\to\infty$. Consequently, in the case $\alpha<1$, for the distribution function $G^{(+)}(x,\alpha,\theta)$ the representation is valid in the form of an infinite series
\begin{equation*}
  G^{(+)}(x,\alpha,\theta)=1-\frac{1}{\pi}\sum_{n=1}^{\infty}\frac{(-1)^{n+1}}{n!}\Gamma(\alpha n)\sin(\tfrac{\pi}{2}\alpha n(1+\theta))x^{-\alpha n},\quad x>0
\end{equation*}

Using the inversion property (\ref{eq:InversionFormula}) for negative $x$ we obtain
\begin{equation*}
  G^{(-)}(x,\alpha,\theta)=   \frac{1}{\pi}\sum_{n=1}^{\infty}\frac{(-1)^{n+1}}{n!}\Gamma(\alpha n)\sin(\tfrac{\pi}{2}\alpha n(1-\theta))(-x)^{-\alpha n},\quad x<0.
\end{equation*}
 If now we introduce the parameter $\theta^*=\theta\sign(x)$ and take the coordinate $x$ in absolute value, then we can combine the last two formulas. As a result, we get the formula (\ref{eq:G_a<1}). This proves the first item of the corollary.

Now we consider the case $\alpha=1$ and $x>1$. As it follows from the expression (\ref{eq:SeriesTermLim}), in this case the series (\ref{eq:G_NInf}) is convergent. It also follows from the expression (\ref{eq:RN_lim}) that $\lim_{N\to\infty}\mathcal{R}_N^\infty(x,\alpha,\theta)=0$. Therefore, the right part of the expression (\ref{eq:G(+)-GN}) is an element of an infinitesimal sequence. In turn, this means that for any fixed $x>1$ the sequence $1-G_N^\infty(x,1,\theta)$ converges to the distribution function $G^{(+)}(x,1,\theta)$. Therefore, the representation in the form of an infinite series is valid for the distribution function $G^{(+)}(x,1,\theta)$
\begin{equation}\label{eq:G(+)_x_expan}
  G^{(+)}(x,1,\theta)=1-\frac{1}{\pi}\sum_{n=1}^{\infty}\frac{(-1)^{n+1}}{n}\sin\left(\tfrac{\pi}{2}n(1+\theta)\right)x^{-n},\quad x>1.
\end{equation}

Using the formula (\ref{eq:InversionFormula}) for negative $x$ we obtain the formula
\begin{equation*}
  G^{(-)}(x,1,\theta)=\frac{1}{\pi}\sum_{n=1}^{\infty}\frac{(-1)^{n+1}}{n}\sin\left(\tfrac{\pi}{2}n(1-\theta)\right)(-x)^{-n},\quad x<-1.
\end{equation*}
    Using now the notation $\theta^*=\theta\sign(x)$ and taking the coordinate $x$ in absolute value we can combine the last two formulas in one expression. As a result, we obtain the expression (\ref{eq:G_a=1}). This proves the second item of the corollary.

Now we consider the case $\alpha>1$. From the expression (\ref{eq:SeriesTermLim}) it follows that in this case the series (\ref{eq:G_NInf}) is divergent at $N\to\infty$. However, from the expression (\ref{eq:RG_NInf}) it is clear that at some fixed  $N$  the estimate is valid
\begin{equation*}
  \mathcal{R}_N^\infty(x,\alpha,\theta)=O(x^{-\alpha N}),\quad x\to\infty.
\end{equation*}
Thus, for each fixed $N$ at $x>0$ from the expression (\ref{eq:G(+)_1}) we obtain
\begin{equation*}
  G^{(+)}(x,\alpha,\theta)=1-\frac{1}{\pi}\sum_{n=1}^{N-1}\frac{(-1)^{n+1}}{n!}\Gamma(\alpha n)\sin(\tfrac{\pi}{2}\alpha n(1+\theta))x^{-\alpha n}+O\left(x^{-\alpha N}\right),\quad x\to\infty.
\end{equation*}
Using the formula (\ref{eq:InversionFormula}) we obtain the representation for negative $x$
\begin{equation*}
  G^{(-)}(x,\alpha,\theta)=\frac{1}{\pi}\sum_{n=1}^{N-1}\frac{(-1)^{n+1}}{n!}\Gamma(\alpha n)\sin(\tfrac{\pi}{2}\alpha n(1-\theta))(-x)^{-\alpha n}+O\left((-x)^{-\alpha N}\right),\quad x\to-\infty.
\end{equation*}
Introducing now the notation $\theta^*=\theta\sign(x)$ and taking the coordinate $x$ in absolute value we can combine the last two expressions into one. As a result, we obtain
\begin{equation*}
  G(x,\alpha,\theta)=\tfrac{1}{2}(1+\sign(x))-\frac{\sign(x)}{\pi}\sum_{n=1}^{N-1}\frac{(-1)^{n+1}}{n!}\Gamma(\alpha n)\sin(\tfrac{\pi}{2}\alpha n(1+\theta^*))|x|^{-\alpha n}+O\left(|x|^{-\alpha N}\right),\quad x\to\pm\infty.
\end{equation*}
Thus, we have obtained the definition of an asymptotic series.  Therefore, this expression can be written as (\ref{eq:G_a>1}). This proves the third item of the corollary.
It should be noted that we do not consider the case $\alpha=1$ and $|x|<1$ since in this case the series (\ref{eq:G_NInf}) diverges. Thus, the corollary is completely proved.
\begin{flushright}
  $\Box$
\end{flushright}

\end{proof}
It should be noted that the case $\alpha=1$ is related to one of those few cases when both the probability density and the distribution function are expressed in terms of elementary functions. In this case, the distribution function is expressed by the formula (\ref{eq:G(x)_a=1}). The derivation of this formula, as well as the proof of corollary~\ref{corol:SSL_cdf} (see Appendix~\ref{sec:IntRepr})  can be found in the paper\cite{Saenko2020b}. In the article \cite{Saenko2023} it was shown that the expansion of the probability density in a series in the case  $\alpha=1$ at $x\to\infty$  converges to the density $g(x,1,\theta)=\frac{\cos(\pi\theta/2)}{\pi(x^2-2x\sin(\pi\theta/2)+1)}$ at $N\to\infty$. Similarly, for the distribution function one can show that the expansion (\ref{eq:G_a=1}) converges to the distribution function (\ref{eq:G(x)_a=1}) at $|x|>1$. We formulate this result in the form of a remark.
\begin{remark}
In the case $\alpha=1$ for any $-1<\theta<1$ in the domain $|x|>1$ the series (\ref{eq:G_a=1}) converges to the distribution function (\ref{eq:G(x)_a=1}).
\end{remark}

\begin{proof}
To prove this, we will consider the distribution function (\ref{eq:G(x)_a=1}) and show that the expansion of this distribution function in a Taylor series at $x\to\infty$ has the form (\ref{eq:G_a=1}). Using the reduction formulas $\cos\left(\tfrac{\pi}{2}\theta\right)=\sin\left(\tfrac{\pi}{2}+\tfrac{\pi}{2}\theta\right)$, $\sin\left(\tfrac{\pi}{2}\theta\right)=-\cos\left(\tfrac{\pi}{2}+\tfrac{\pi}{2}\theta\right)$, we will write down the distribution function  (\ref{eq:G(x)_a=1}) in the form
\begin{equation}\label{eq:G(x)_a=1_tmp0}
G(x,1,\theta)=\frac{1}{2}+\frac{1}{\pi}\arctan\left(\frac{x+\cos\left(\tfrac{\pi}{2}(1+\theta)\right)}{\sin\left(\tfrac{\pi}{2}(1+\theta)\right)}\right)
\end{equation}

Further, since we need to obtain the expansion of the distribution function in a Taylor series at  $x\to\infty$, then in this expression we will substitute the variable $x=1/y$ in this expression and start considering the case $x\geqslant0$. The relationship obtained in this way we will denote as $G^{(+)}(y,1,\theta)$. As a result, we obtain
\begin{equation}\label{eq:G(y)+_a=1}  G^{(+)}(y,1,\theta)=\frac{1}{2}+\frac{1}{\pi}\arctan\left(\frac{\tfrac{1}{y}+\cos\left(\tfrac{\pi}{2}(1+\theta)\right)}
  {\sin\left(\tfrac{\pi}{2}(1+\theta)\right)}\right),\quad y\geqslant0.
\end{equation}
Hence it is clear that the behavior of the function (\ref{eq:G(y)+_a=1}) at $y\to0$ corresponds to the behavior of the function (\ref{eq:G(x)_a=1_tmp0}) at $x\to\infty$. Therefore, expanding the function (\ref{eq:G(y)+_a=1}) in a Taylor series in the vicinity of the point $y=0$ and getting back to the variable $x$, we obtain the expansion of the distribution function (\ref{eq:G(x)_a=1_tmp0}) into a power series at $x\to\infty$.

We will take into account that the function $\arctan(x)$ is infinitely differentiable. Consequently, the expansion of the function (\ref{eq:G(y)+_a=1}) in a Taylor series in the vicinity of the point $y=0$ has the form
\begin{equation}\label{eq:G(y)_Teylor_expan}
  G^{(+)}(y,1,\theta)=G^{(+)}(0,1,\theta)+\sum_{n=1}^{\infty}\frac{1}{n!}\left.\frac{d^n G^{(+)}(y,1,\theta)}{dy^n}\right|_{y=0}y^n.
\end{equation}

At the beginning, we will calculate the first derivative of the function (\ref{eq:G(y)+_a=1}). We obtain
\begin{equation*}
  \frac{d G^{(+)}(y,1,\theta)}{dy}=-\frac{\sin\left(\tfrac{\pi}{2}(1+\theta)\right)}
  {\pi\left(y^2+2y\cos\left(\tfrac{\pi}{2}(1+\theta)\right)+1\right)}.
\end{equation*}
We represent this expression in the form
\begin{equation*}
   \frac{d G^{(+)}(y,1,\theta)}{dy}=-\frac{\sin\left(\tfrac{\pi}{2}(1+\theta)\right)}{\pi}f(g(y)),
\end{equation*}
where
\begin{equation}\label{eq:f_g_def}
 f\equiv f(g)=1/g,\quad g\equiv g(y)=y^{2}+2y\cos\left(\tfrac{\pi}{2}(1+\theta)\right)+1.
\end{equation}

Thus, for the derivative of the series of $n$ of the function $G^{(+)}(y,1,\theta)$ we get
\begin{multline*}
  \frac{d^n G^{(+)}(y,1,\theta)}{dy^n}=\frac{d^{n-1}}{dy^{n-1}}\frac{d G^{(+)}(y,1,\theta)}{dy}
  =-\frac{\sin\left(\tfrac{\pi}{2}(1+\theta)\right)}{\pi} \frac{d^{n-1} f(g(y))}{dy^{n-1}}\\
  =-\frac{\sin\left(\tfrac{\pi}{2}(1+\theta)\right)}{\pi} \sum_{k=0}^{\left[\tfrac{n-1}{2}\right]} \frac{(-1)^{n-1-k} (n-1)! (n-1-k)!}{k!(n-1-2k)!}\frac{\left(2y+2\cos\left(\tfrac{\pi}{2}(1+\theta)\right)\right)^{n-1-2k}} {\left(y^2+2y\cos\left(\tfrac{\pi}{2}(1+\theta)\right)+1\right)^{n-k}},
\end{multline*}
 where the formula was used
\begin{equation*}
     \frac{d^n f(g(y))}{dy^n}=\sum_{k=0}^{\left[\tfrac{n}{2}\right]} \frac{(-1)^{n-k} n! (n-k)!}{k!(n-2k)!}\frac{\left(2y+2\cos\left(\tfrac{\pi}{2}(1+\theta)\right)\right)^{n-2k}} {\left(y^2+2y\cos\left(\tfrac{\pi}{2}(1+\theta)\right)+1\right)^{n-k+1}},
\end{equation*}
which was obtained in the article \cite{Saenko2023}.

 At the point $y=0$ this derivative has the value
\begin{multline}\label{eq:G(+)_nder_y=0}
  \left.\frac{d^n G^{(+)}(y,1,\theta)}{dy^n}\right|_{y=0}
  = -\frac{\sin\left(\tfrac{\pi}{2}(1+\theta)\right)}{\pi} (-1)^{n-1} (n-1)!\\
  \times \sum_{k=0}^{\left[\tfrac{n-1}{2}\right]} \frac{(-1)^{k} (n-k-1)!}{k!(n-2k-1)!}\left(2\cos\left(\tfrac{\pi}{2}(1+\theta)\right)\right)^{n-2k-1}
  =\frac{(n-1)!}{\pi}(-1)^{n-1}\sin\left(\tfrac{\pi}{2}n(1+\theta)\right).
\end{multline}
Here it was taken into account that $(-1)^{-k}=(-1)^k$, and also the formula was used for $\sin(n\varphi)$ (see, for example, \cite{Schaum2018})
\begin{equation*}
     \sin(n\varphi)=\sin\varphi\sum_{k=0}^{\left[\frac{n-1}{2}\right]}(-1)^{k}\frac{(n-k-1)!}{k!(n-2k-1)!}(2\cos\varphi)^{n-2k-1}.
\end{equation*}

Now we substitute the expression (\ref{eq:G(+)_nder_y=0}) in the formula (\ref{eq:G(y)_Teylor_expan}) and take into account that  $G^{(+)}(0,1,\theta)=\frac{1}{2}+\arctan(\infty)=1$. As a result, the expression (\ref{eq:G(y)_Teylor_expan}) becomes
\begin{equation*}
  G^{(+)}(y,1,\theta)=1-\frac{1}{\pi}\sum_{n=1}^{\infty}\frac{(-1)^{n-1}}{n}\sin\left(\tfrac{\pi}{2}(1+\theta)\right) y^n.
\end{equation*}
Thus, we have obtained the expansion of the function (\ref{eq:G(y)+_a=1}) in a Taylor series in the vicinity of the point $y=0$.

To obtain the expansion of the distribution function (\ref{eq:G(x)_a=1_tmp0}) at $x\to\infty$, one must return to the variable $x$ in the last expression. By substituting the variable $y=1/x$, we obtain
\begin{equation*}
  G^{(+)}(x,1,\theta)=1-\frac{1}{\pi}\sum_{n=1}^{\infty}\frac{(-1)^{n-1}}{n}\sin\left(\tfrac{\pi}{2}(1+\theta)\right) x^{-n}.
\end{equation*}
This expression is the expansion of the distribution function (\ref{eq:G(x)_a=1_tmp0}) at $x\to\infty$. As one can see, it completely coincides with the expansion obtained earlier (\ref{eq:G(+)_x_expan}). Corollary~\ref{corol:cdfConverg} shows that this series converges at $n\to\infty$ in the domain $x>1$. Hence
\begin{equation*}
  G^{(+)}(x,1,\theta)=1-\frac{1}{\pi}\sum_{n=1}^{\infty}\frac{(-1)^{n-1}}{n}\sin\left(\tfrac{\pi}{2}(1+\theta)\right) x^{-n}, \quad x>1.
\end{equation*}

To obtain the expansion of the distribution function for negative $x$ we use the inversion property and, in particular, the formula (\ref{eq:InversionFormula}). As a result, we get
\begin{equation*}
 G^{(-)}(x,1,\theta)=\frac{1}{\pi}\sum_{n=1}^{\infty}\frac{(-1)^{n+1}}{n}\sin\left(\tfrac{\pi}{2}n(1-\theta)\right)(-x)^{-n},\quad x<-1.
\end{equation*}
If we now introduce the parameter $\theta^*=\theta\sign(x)$ and take the variable $x$ in absolute value then it is possible to combine the formulas for $G^{(+)}(x,1,\theta)$ and  $G^{(-)}(x,1,\theta)$ into one formula. As a result, we obtain
\begin{equation*}
  G(x,1,\theta)=\tfrac{1}{2}(1+\sign(x))-\frac{\sign(x)}{\pi}\sum_{n=1}^{\infty}\frac{(-1)^{n+1}}{n} \sin\left(\tfrac{\pi}{2}n(1+\theta^*)\right)|x|^{-n},\quad |x|>1.
\end{equation*}
This formula coincides completely with the expression (\ref{eq:G_a=1}). Thus, the expansion of the distribution function (\ref{eq:G(x)_a=1_tmp0}) coincides exactly with the expansion (\ref{eq:G_a=1})which was obtained earlier. Therefore, in the domain $|x|>1$ the series (\ref{eq:G_a=1}) converges to the distribution function (\ref{eq:G(x)_a=1}). The remark has been proved.
\begin{flushright}
  $\Box$
\end{flushright}
\end{proof}

Theorem~\ref{theor:cdfExpansion} gives an opportunity to calculate the distribution function for large values of the coordinate $x$ using a power series (\ref{eq:G(x)}). As mentioned in the Introduction, for large values of the coordinate $x$ the use of the integral representation (\ref{eq:G(x)_int}) no longer allows one to calculate the distribution function correctly, and here it is necessary to apply other calculation methods.  The most obvious and probably the only way to calculate the distribution function at large values of $x$ is to use the expansions (\ref{eq:G(x)}).

However, before we use this expansion, it is necessary to obtain a criterion that makes it possible to determine the coordinate $x$ in which, for a specified value of $\alpha$ and the number of summands $N$ in the sum (\ref{eq:G(x)}), the specified accuracy of calculating the distribution function will be achieved. Such a criterion can be obtained by evaluating the remainder term (\ref{eq:RG_NInf}). From the formula (\ref{eq:G(x)}) and (\ref{eq:RG_NInf}) it follows that
\begin{equation*}
  \left|G(x,\alpha,\theta) -\tfrac{1}{2}(1+\sign(x))+\sign(x) G_N^\infty(|x|,\alpha,\theta^*)\right|\leqslant
  \frac{|x|^{-\alpha N}}{\alpha N!}\left(\Gamma(\alpha N)+|x|^{-\alpha}\Gamma(\alpha (N+1))\right).
\end{equation*}
If now for the specified $\alpha$ and $N$ we give the value of the absolute error $\varepsilon$, the value of which should not exceed the true absolute error of the calculation using the expansion (\ref{eq:G(x)}), i.e.
\begin{equation*}
   \left|G(x,\alpha,\theta) -\tfrac{1}{2}(1+\sign(x))+\sign(x) G_N^\infty(|x|,\alpha,\theta^*)\right|\leqslant\varepsilon,
\end{equation*}
then this gives an opportunity to introduce the threshold coordinate $x_\varepsilon^N$. The value of the threshold coordinate can be found from the solution to the equation
\begin{equation}\label{eq:x_eps_eq_cdf}
  \frac{\left|x_\varepsilon^N\right|^{-\alpha N}}{\alpha N!}\left(\Gamma(\alpha N)+\left|x_\varepsilon^N\right|^{-\alpha}\Gamma(\alpha (N+1))\right)=\varepsilon.
\end{equation}

Unfortunately, it is impossible to solve this equation analytically and obtain an explicit expression for the threshold coordinate $x_\varepsilon^N$. Nevertheless, numerical methods without much difficulty help to find a solution to this equation for specified $\alpha$, $N$ and $\varepsilon$.
As a result, we obtain the condition
\begin{equation}\label{eq:|G-GN|_cond}
  \left|G(x,\alpha,\theta) -\tfrac{1}{2}(1+\sign(x))+\sign(x) G_N^\infty(|x|,\alpha,\theta^*)\right|\leqslant\varepsilon,\quad |x|\geqslant x_\varepsilon^N.
\end{equation}
It means that at the values of $|x|\geqslant x_\varepsilon^N$ the absolute error in calculating the distribution function using a power series (\ref{eq:G(x)}) will not exceed the specified accuracy level $\varepsilon$. Here $\alpha, N$ and $\varepsilon$ are specified and $x_\varepsilon^N$ is found from the solution to the equation (\ref{eq:x_eps_eq_cdf}).

Taking account of all the foregoing, we obtain formulas for calculating the distribution function
\begin{equation}\label{eq:cdfCalcExpr}
  G_N(x,\alpha,\theta) =\tfrac{1}{2}(1+\sign(x))-\sign(x) G_N^\infty(|x|,\alpha,\theta^*), \quad |x|\geqslant x_\varepsilon^N.
\end{equation}
Here $G_N^\infty(x,\alpha,\theta)$ is determined by the expression (\ref{eq:G_NInf}), and the threshold coordinate $x_\varepsilon^N$ is found from the solution to the equation (\ref{eq:x_eps_eq_cdf}). The accuracy level $\varepsilon$ and the number of summands $N$  are specified beforehand. In this case, it can be guaranteed that in the domain $|x|\geqslant x_\varepsilon^N$ at a specified $N$ the absolute error in calculating the distribution function using this formula will not exceed $\varepsilon$.

Figures~\ref{fig:cdf_a07},~\ref{fig:cdf_a1},~and~\ref{fig:cdf_a13} show the calculation results  of the distribution function using the formula (\ref{eq:cdfCalcExpr}) and the results of calculating the quantity  of the absolute error for the values $\alpha=0.7, 1, 1.3$. Figures~\ref{fig:cdf_a07}a,~\ref{fig:cdf_a1}a,~and~\ref{fig:cdf_a13}a demonstrate the results of calculating the distribution function. In these figures, the solid curve corresponds to the exact value of the distribution function. In the case $\alpha\neq1$ the integral representation (\ref{eq:G(x)_int}) was used for calculations, and in the case $\alpha=1$ the formula (\ref{eq:G(x)_a=1}) was used. The dashed-dotted curves in these figures correspond to the results of the calculation using the expansion (\ref{eq:cdfCalcExpr}) for the values $N=3, 10, 30, 60, 90$. The circles in these figures show the position of the threshold coordinate $x_\varepsilon^N$ for the chosen values of $N$ and the specified accuracy level $\varepsilon=10^{-5}$. The value of the threshold coordinate for each $N$ and selected $\varepsilon$ was found by solving the equation (\ref{eq:x_eps_eq_cdf}).

Figures~\ref{fig:cdf_a07}b,~\ref{fig:cdf_a1}b,~and~\ref{fig:cdf_a13}b give the absolute error of calculating the distribution function using the expansion (\ref{eq:cdfCalcExpr}). In these figures, the solid curve is the exact value of the absolute error $|G(x,\alpha,\theta)-G_N(x,\alpha,\theta)|$. Here to calculate $G(x,\alpha,\theta)$ the integral representation (\ref{eq:G(x)_int}) was used in the case $\alpha\neq1$, and the formula (\ref{eq:G(x)_a=1}) was applied in the case $\alpha=1$. To calculate $G_N(x,\alpha,\theta)$ the formula (\ref{eq:cdfCalcExpr}) was used. The dashed-dotted curves are the estimate of the remainder term (\ref{eq:RG_NInf}). The dotted line shows the position of the specified accuracy level $\varepsilon$, and the circles are the position of the threshold coordinate $x_\varepsilon^N$ for each value of $N$.
\begin{figure}
\includegraphics[width=0.47\textwidth]{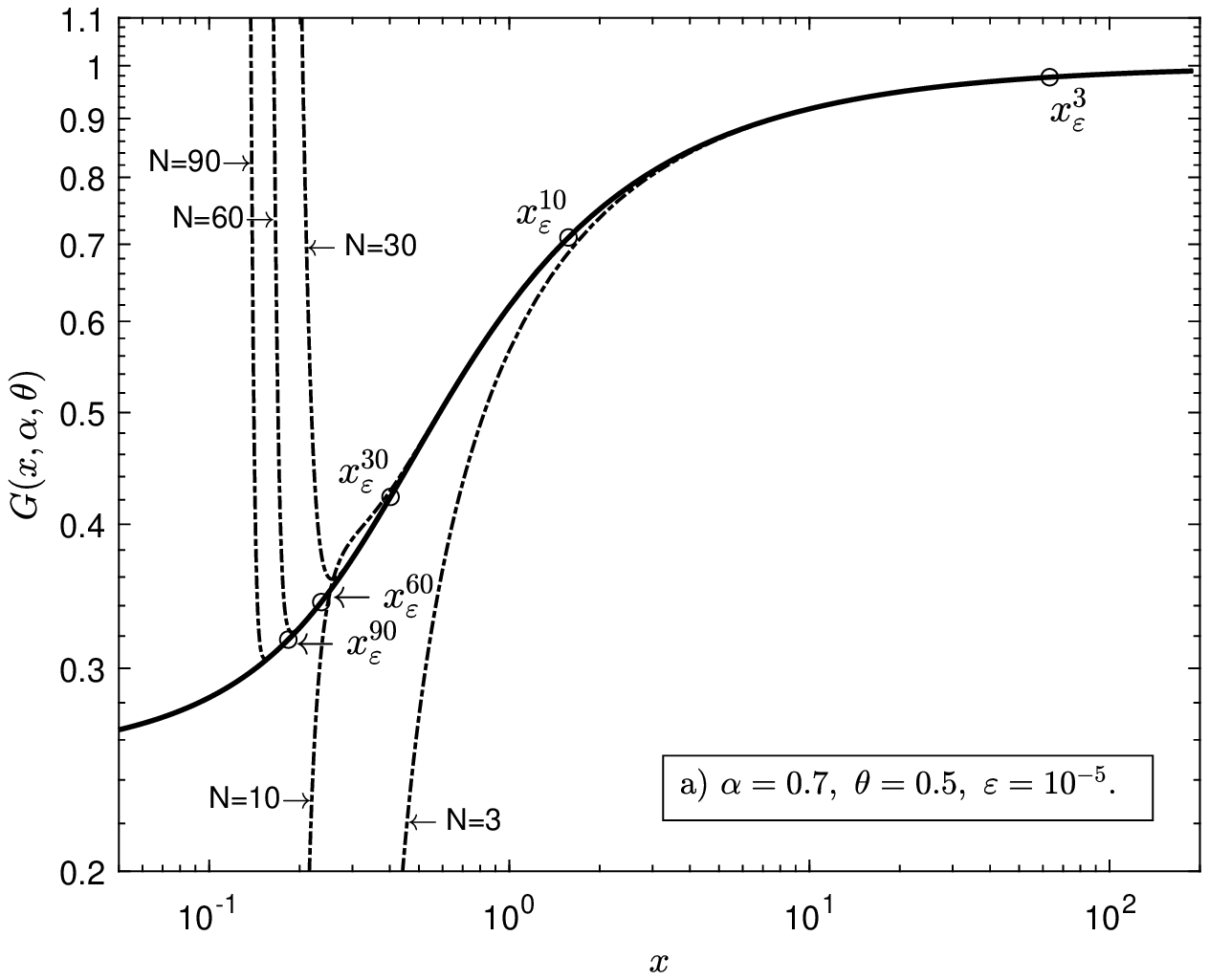}\hfill
\includegraphics[width=0.47\textwidth]{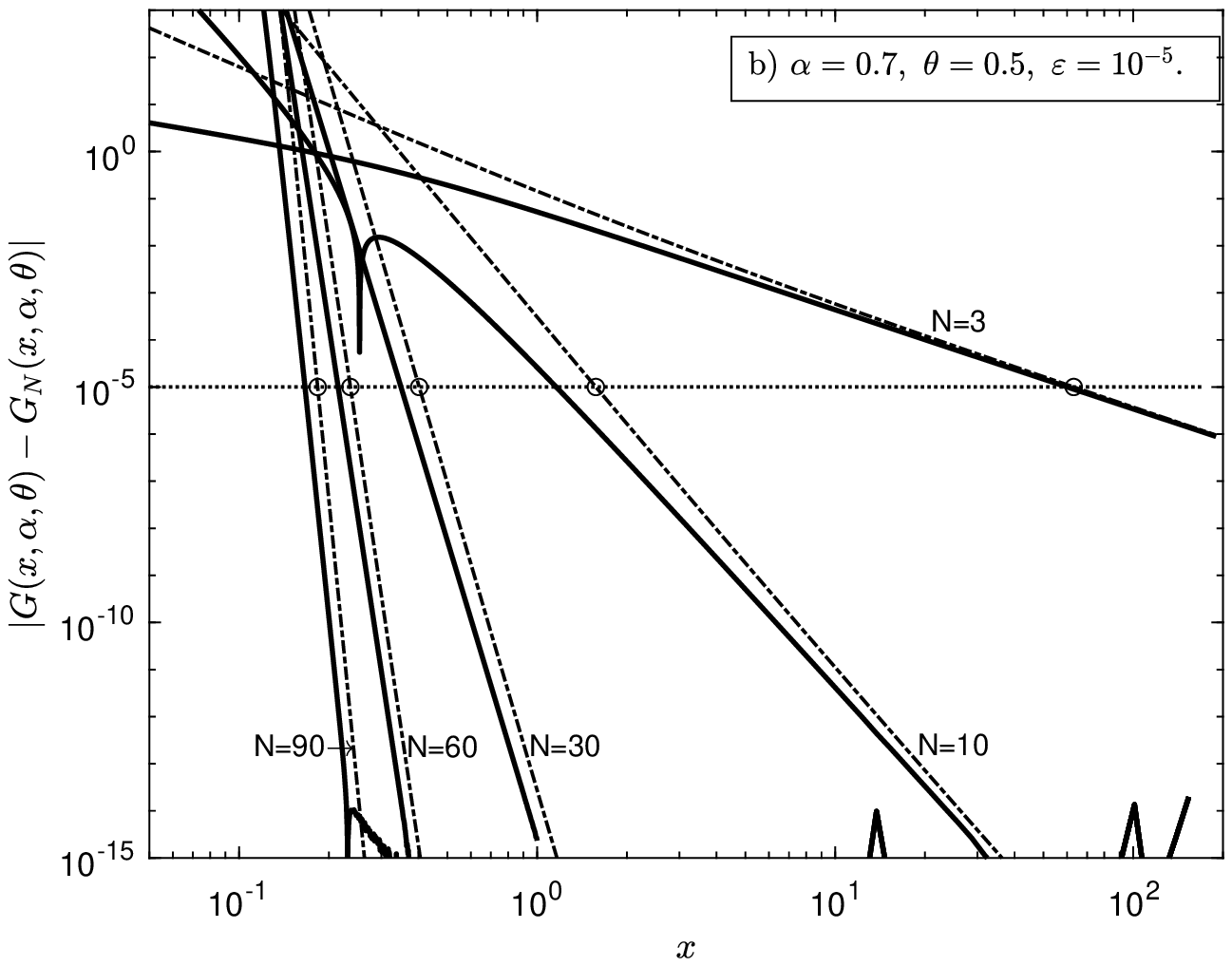}\\
\caption{a) The distribution function $G(x,\alpha,\theta)$ for the parameter values shown in the figure.  The solid curve is the integral representation (\ref{eq:G(x)_int}), the dash-dotted curves are the representation in the form of a power series (\ref{eq:cdfCalcExpr}) for different values of the number of summands $N$ in the sum. The circles show the position of the threshold coordinate  $x_\varepsilon^N$ for corresponding values of $N$ and the specified accuracy level $\varepsilon$. (b) Graph of the absolute error of calculating the distribution function $G(x,\alpha,\theta)$ with the use of a power series (\ref{eq:cdfCalcExpr}). The solid curves are the exact value of the absolute error $|G(x,\alpha,\theta)-G_N(x,\alpha,\theta)|$, the dash-dotted curves are the residual term estimate (\ref{eq:RG_NInf}), the dotted line is the specified accuracy level $\varepsilon$, the circles show the position of the threshold coordinate $x_\varepsilon^N$}\label{fig:cdf_a07}
\end{figure}

\begin{figure}
\includegraphics[width=0.47\textwidth]{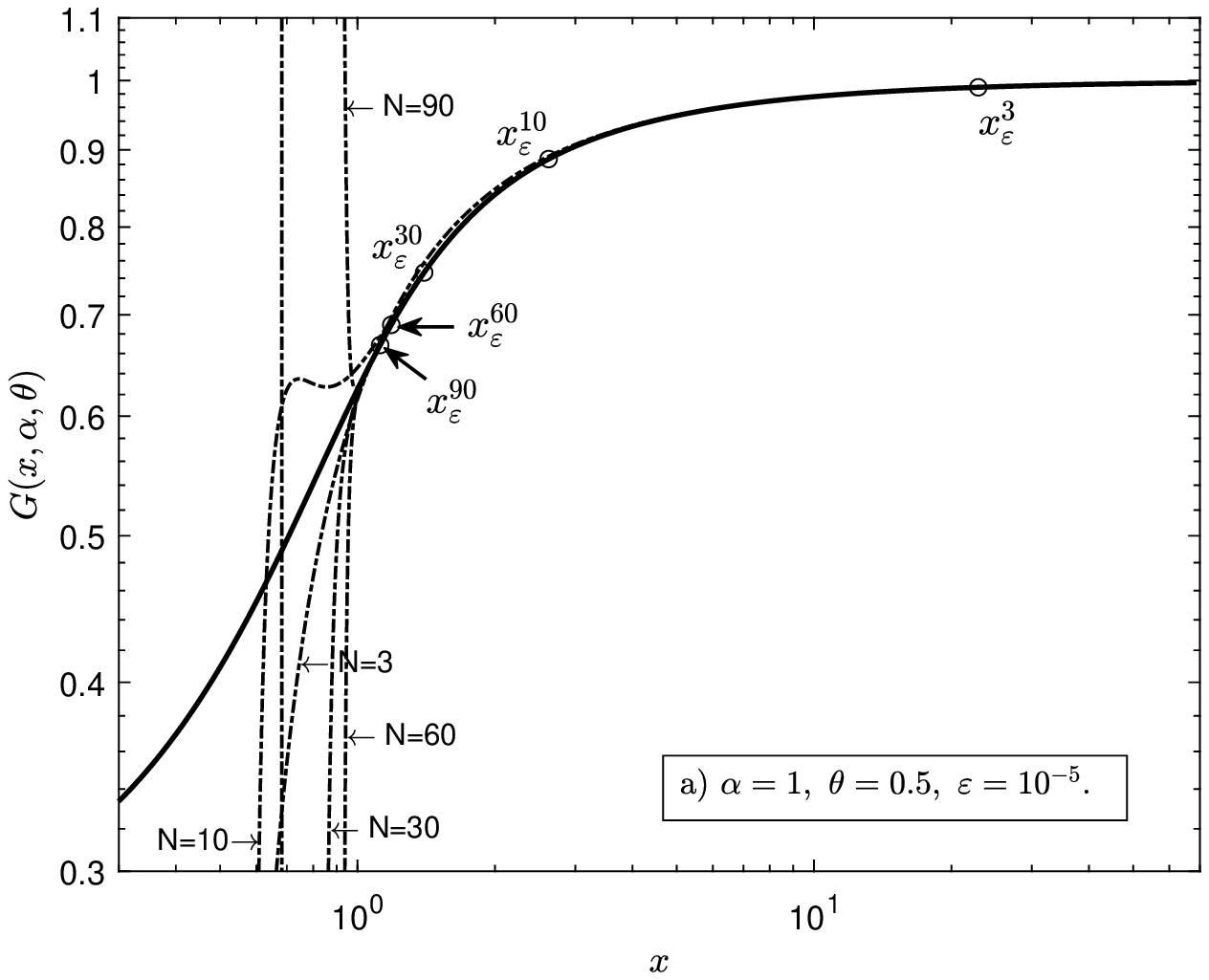}\hfill
\includegraphics[width=0.47\textwidth]{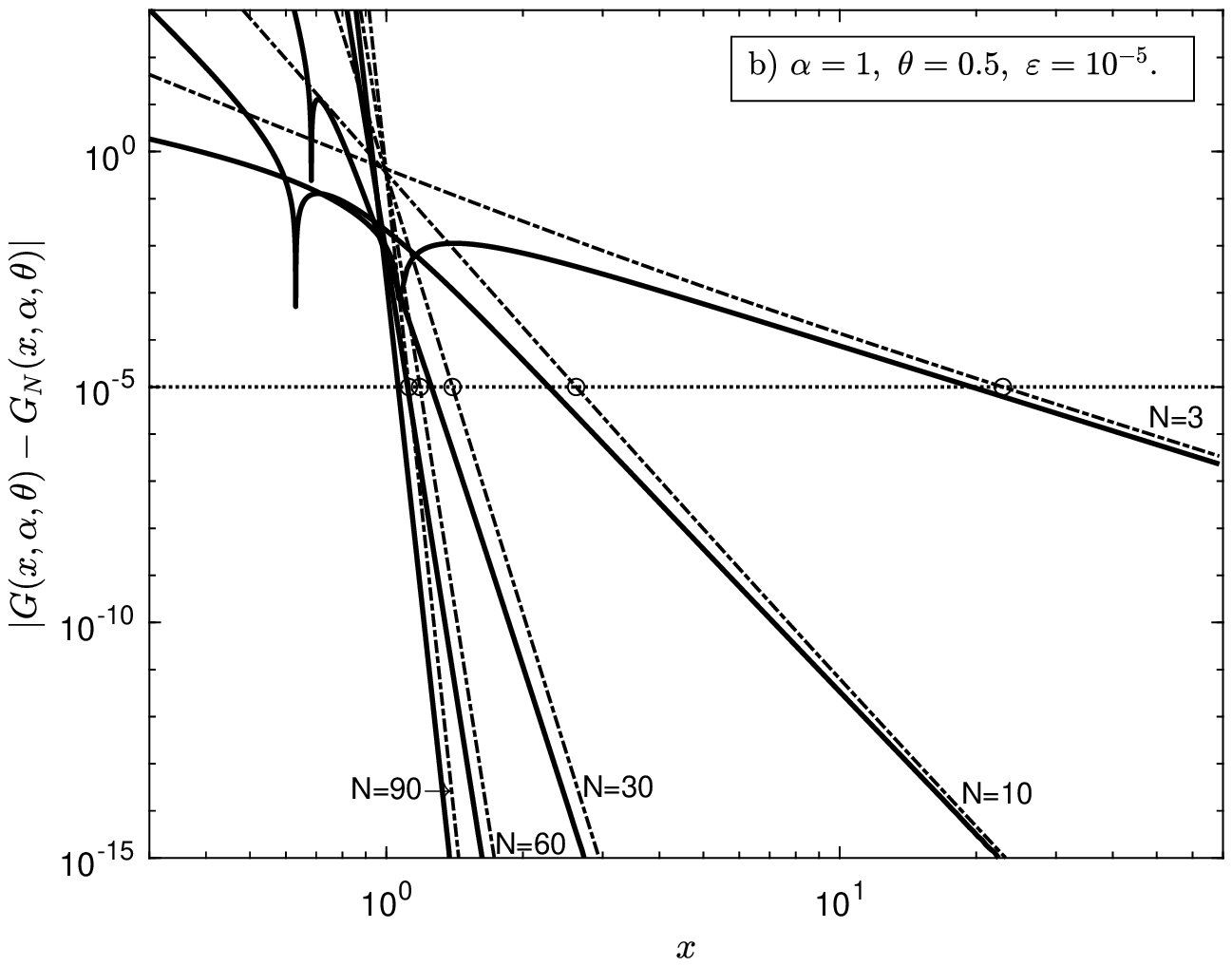}\\
\caption{a) The distribution function $G(x,\alpha,\theta)$ for the parameter values shown in the figure.  The solid curve is the formula (\ref{eq:G(x)_a=1}), the dash-dotted curves are the power series representation (\ref{eq:cdfCalcExpr}) for different values of the number of summands $N$ in the sum. The circles show the position of the threshold coordinate $x_\varepsilon^N$ for the corresponding values of  $N$ and the specified level of accuracy $\varepsilon$. (b) Graph of the absolute error of calculating the distribution function $G(x,\alpha,\theta)$ with the use of a power series  (\ref{eq:cdfCalcExpr}). The solid curves are the exact value of the absolute error $|G(x,\alpha,\theta)-G_N(x,\alpha,\theta)|$, the dash-dotted curves are the residual term estimate (\ref{eq:RG_NInf}), the dotted line is the specified accuracy level $\varepsilon$, the circles show the position of the threshold coordinate $x_\varepsilon^N$}\label{fig:cdf_a1}
\end{figure}

\begin{figure}
\includegraphics[width=0.47\textwidth]{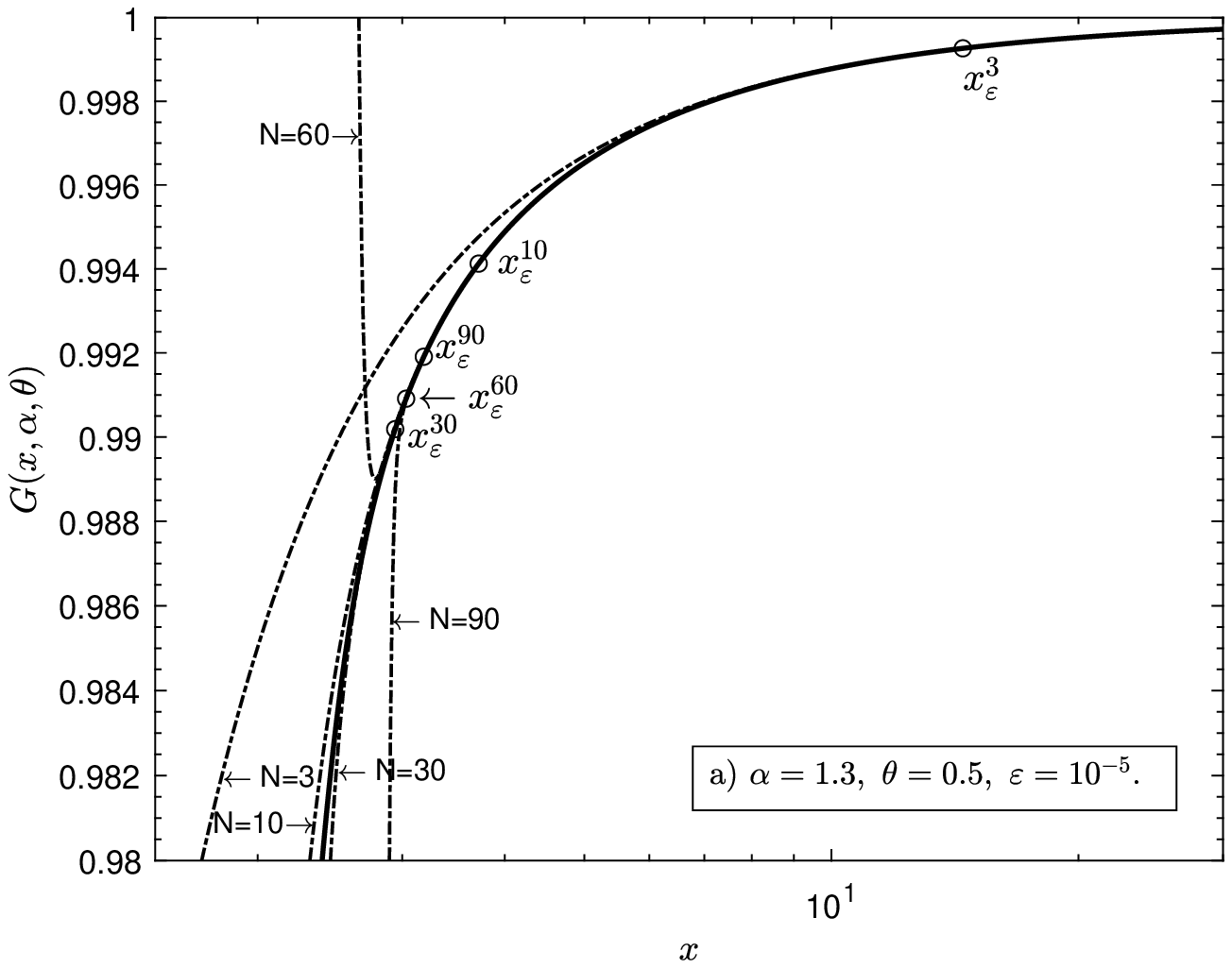}\hfill
\includegraphics[width=0.47\textwidth]{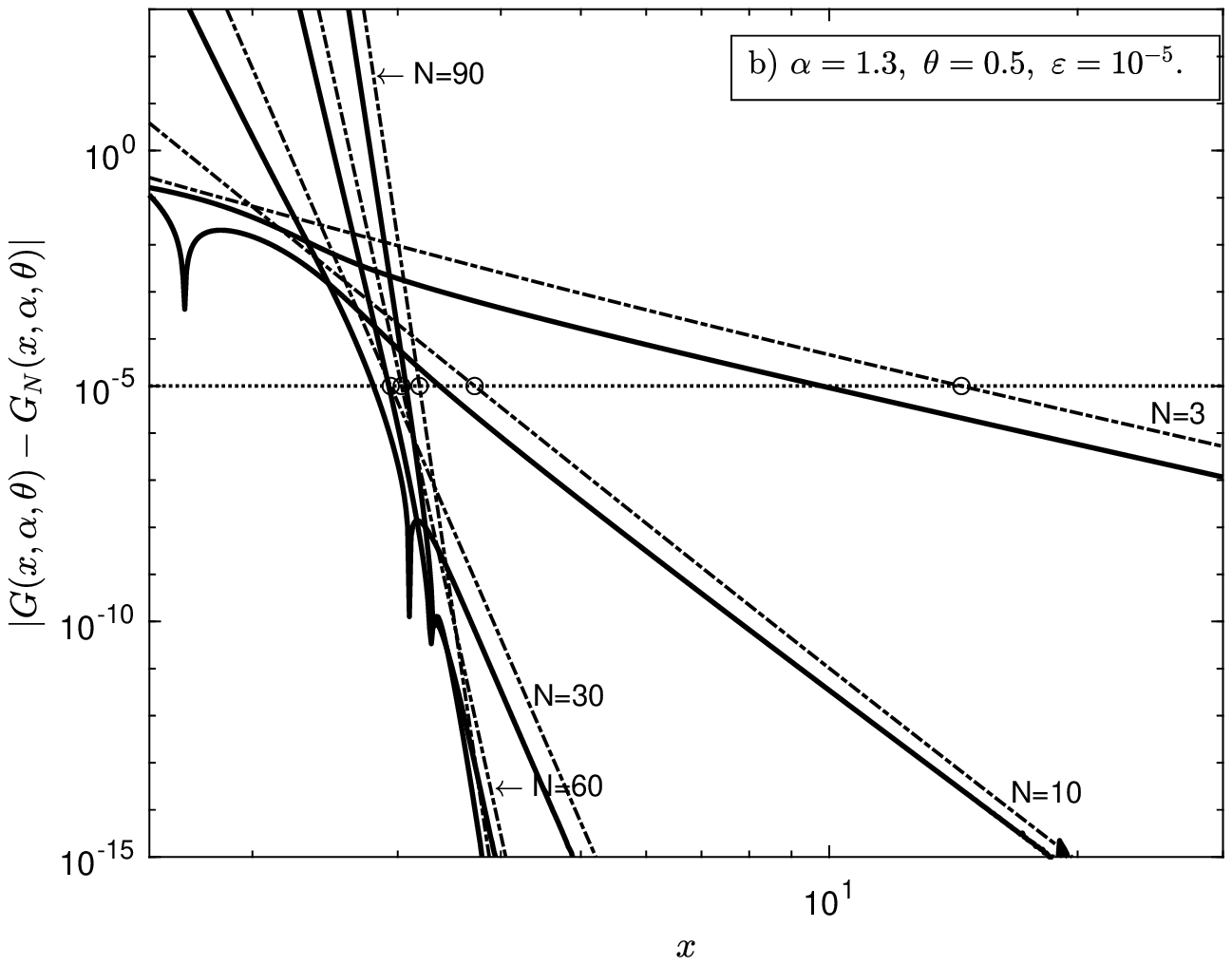}\\
\caption{a) The distribution function $G(x,\alpha,\theta)$ for the parameter values shown in the figure. The solid curve is the integral representation (\ref{eq:G(x)_int}), the dash-dotted curves are the representation in a power series (\ref{eq:cdfCalcExpr}) for different values of the number of summands $N$ in the sum. The circles show the position of the threshold coordinate $x_\varepsilon^N$ for the corresponding values $N$ and the specified accuracy level $\varepsilon$. (b) Graph of the absolute error of calculating the distribution function $G(x,\alpha,\theta)$ using a power series (\ref{eq:cdfCalcExpr}). The solid curves are the exact value of the absolute error $|G(x,\alpha,\theta)-G_N(x,\alpha,\theta)|$, the dash-dotted curves are the residual term estimate (\ref{eq:RG_NInf}), the dotted line is the specified accuracy level $\varepsilon$, the circles show the position of the threshold coordinate  $x_\varepsilon^N$}\label{fig:cdf_a13}
\end{figure}

Figures~\ref{fig:cdf_a07}b,~\ref{fig:cdf_a1}b,~and~\ref{fig:cdf_a13}b clearly show that in the domain $x>x_\varepsilon^N$ the exact value of the absolute error (solid curves) turns out to be less than the estimate of the remainder term (\ref{eq:RG_NInf}) (dashed-dotted curves). It is also clearly seen that in the domain $x>x_\varepsilon^N$, both the exact value of the absolute error and the residual term estimate are less than the value of the selected accuracy level $\varepsilon$. This means that in the domain $|x|>x_\varepsilon^N$ the formula (\ref{eq:cdfCalcExpr}) can be used to calculate the distribution function. At the same time, it can be guaranteed that the absolute error in calculating the distribution function using this formula will not exceed the selected accuracy level $\varepsilon$, and in reality it will be much less than this value.

If we now analyze the behavior of the threshold coordinate $x_\varepsilon^N$ from the number of terms $N$, then we can see that for each of the considered cases $\alpha<1$, $\alpha=1$ and $\alpha>1$ this behavior differs. It can be seen from fig.~\ref{fig:cdf_a07} that in the case $\alpha=0.7$ as $N$ increases the value $x_\varepsilon^N$ decreases. This behavior of the threshold coordinate is the result of corollary~\ref{corol:cdfConverg}. Indeed, in the first item of this corollary it is proved that in the case $\alpha<1$ at $N\to\infty$ the series (\ref{eq:G_a<1}) converges for any $x$. This means that as the number of terms in the formula (\ref{eq:cdfCalcExpr}) increases, the accuracy of calculating the distribution function at some fixed point $x$ will increase. In turn, this leads to the fact that the range of coordinates $x$ under which the condition (\ref{eq:|G-GN|_cond}) is satisfied will expand. Therefore, as $N$ increases the value of the threshold coordinate $x_\varepsilon^N$ will decrease.

To calculate the limit value of the coordinate $x_\varepsilon^N$ at $N\to\infty$ we will consider the equation (\ref{eq:x_eps_eq_cdf}) and assume that $N\to\infty$. Taking account that $N+1\approx N$ at $N\to\infty$ this equation takes the form $2|x_\varepsilon^N|^{-\alpha N}\Gamma(\alpha N)=\varepsilon\pi\Gamma(N)$. From here we find
\begin{equation*}
  |x_\varepsilon^N|=\left(\frac{2\Gamma(\alpha N)}{\pi\varepsilon\Gamma(N)}\right)^{\frac{1}{\alpha N}}.
\end{equation*}
We now find the limit of this expression at $N\to\infty$. Using the Stirling formula (\ref{eq:Stirling}), we obtain
\begin{multline}\label{eq:x_eps_lim}
  \lim_{N\to\infty} |x_\varepsilon^N|=\lim_{N\to\infty}\left(\frac{2\Gamma(\alpha N)}{\pi\varepsilon\Gamma(N)}\right)^{\frac{1}{\alpha N}}
  =\lim_{N\to\infty}\left(\frac{2}{\pi\varepsilon}\frac{e^{-\alpha N}(\alpha N)^{\alpha N-\frac{1}{2}}\sqrt{2\pi}}{e^{-N} (N)^{N-\frac{1}{2}}\sqrt{2\pi}}\right)^{\frac{1}{\alpha N}} \\
  =e^{\frac{1}{\alpha}-1} \alpha^{-1}\lim_{N\to\infty}\left(\frac{2}{\pi\varepsilon\alpha^2}\right)^{\frac{1}{\alpha N}} N^{1-\frac{1}{\alpha}}=\begin{cases}
                           0, &  \alpha<1 \\
                           1, & \alpha=1 \\
                           \infty, & \alpha>1.
                         \end{cases}
\end{multline}
Thus, in the case $\alpha<1$ we get $\lim_{N\to\infty} |x_\varepsilon^N|=0$. This result is a consequence of the convergence of the series (\ref{eq:G_NInf}) in the case $\alpha<1$.

A similar behavior of the threshold coordinate is also observed in the case $\alpha=1$. Figure~\ref{fig:cdf_a1} shows that as the number of $N$ summands in the formula (\ref{eq:cdfCalcExpr}) increases the value of the threshold coordinate decreases. However, unlike the previous case, it follows from the expression (\ref{eq:x_eps_lim}) that in this case $\lim_{N\to\infty}|x_\varepsilon^N|=1$. Such behavior of the threshold coordinate is the result of  corollary~\ref{corol:cdfConverg}, where in the second item it was proved that in the case $\alpha=1$ the series (\ref{eq:G_a=1}) converges at $N\to\infty$ in the domain $|x|>1$.

In the case $\alpha>1$ the behavior of the threshold coordinate changes. Figure~\ref{fig:cdf_a13}a shows that as the number of summands $N$ in the formula (\ref{eq:cdfCalcExpr}) increases, the value of the threshold coordinate first decreases: $x_\varepsilon^3>x_\varepsilon^{10}>x_\varepsilon^{30}$. However, a further increase in $N$ leads to an increase in the threshold coordinate: $x_\varepsilon^{30}<x_\varepsilon^{60}<x_\varepsilon^{90}$. Such behavior of the threshold is in full accordance with corollary~\ref{corol:cdfConverg}, where in the third item it was proved that in the case $\alpha>1$ the series (\ref{eq:G_NInf}) diverged at $N\to\infty$.  The cause of the divergence of this series lies in the presence of the multiplier $\Gamma(\alpha n)/\Gamma(n+1)$. Hence, it is clear that at  $\alpha>1$ this multiplier is more than 1 and as  $n$ increases the value of this multiplier only increases. Such behavior of this multiplier is the reason for the divergence of the series (\ref{eq:G_NInf}). In this series there is also a multiplier $x^{-\alpha n}$. As one can see, as the value of $x$ increases, the value of this multiplier decreases. The competition between these two factors leads to the observed behavior of the threshold coordinate $x_\varepsilon^N$.

Indeed, the threshold coordinate is found as a result of solving the equation $|G(x,\alpha,\theta)-G_N(x,\alpha,\theta)|=\varepsilon$. Therefore, first the increase in the number of summands $N$ in the sum (\ref{eq:cdfCalcExpr}) leads to a decrease in the coordinate $x_\varepsilon^N$. This is testified by the fact that $x_\varepsilon^3>x_\varepsilon^{10}>x_\varepsilon^{30}$. However, further increase in  $N$ leads to the fact that the factor $\Gamma(\alpha n)/\Gamma(n+1)$ starts growing rapidly, and to compensate for this growth, it is necessary to increase $x$ in the multiplier $x^{-\alpha n}$. This is what leads to a shift of the threshold coordinate towards larger values of $x$ as $N$ increases. Fig.~\ref{fig:cdf_a13} demonstrates such behavior from which it is clear that  $x_\varepsilon^{30}<x_\varepsilon^{60}<x_\varepsilon^{90}$. In case, if $n\to\infty$, then the factor $\Gamma(\alpha n)/\Gamma(n)\to\infty$. Therefore, to compensate for this growth, it is necessary that $x\to\infty$ in the multiplier $x^{-\alpha n}$. Thus, the obtained conclusion is in full accordance with the expression (\ref{eq:x_eps_lim}), which shows that $\lim_{N\to\infty}|x_\varepsilon^N|=\infty$, if $\alpha>1$.

It should be pointed out, if in the considered case ($\alpha>1$) we fix some arbitrary $N$, then in view of the presence of the factor $x^{-\alpha N}$ with an increase in the value of $x$ one can achieve any preset calculation accuracy $\varepsilon$. Consequently, at $\alpha>1$ the formula (\ref{eq:cdfCalcExpr}) is asymptotic at $|x|\to\infty$, which is the result of corollary~\ref{corol:cdfConverg}.

\section{Calculation of the distribution function at $x\to\infty$}

We return to the question of calculating the distribution function of a strictly stable law in the case of large values of the coordinate $x$. The main approach to calculate the distribution function is to use the integral representation (\ref{eq:G(x)_int}). In theory, this integral representation is valid for all values of the parameters $\alpha,\theta$ (except for the value $\alpha=1$) and all $x$. However, in practice, it is not always possible to calculate the integral numerically included in this integral representation. Problems arise at small and large values of the coordinate $x$.  The cause of the difficulties that arise is the behavior of the integrand in the formula (\ref{eq:G(x)_int}).
\begin{figure}
  \centering
  \includegraphics[width=0.8\textwidth]{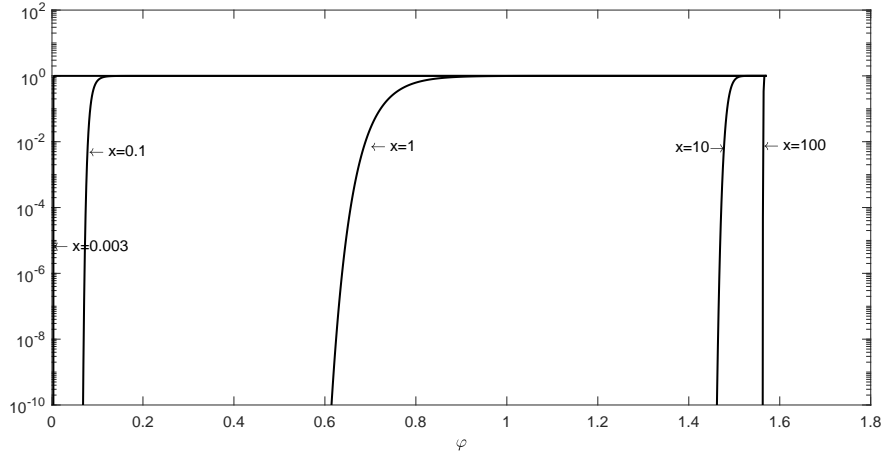}
  \caption{ The relationship between the integrand of the integral representation for the distribution function (\ref{eq:G(x)_int}) and the value of the integration variable $\varphi$. The figure shows graphs of the integrand for the values of the parameters $\alpha=1.1, \theta=0$ and the specified values of the coordinate $x$}\label{fig:cdfIntegrand}
\end{figure}

Fig.~\ref{fig:cdfIntegrand} presents a graph of the integrand of the integral representation (\ref{eq:G(x)_int}) for the parameters $\alpha=1.1, \theta=0$ and the specified values of coordinates $x$ depending on the integration variable $\varphi$.   The variable $\varphi$ changes within the range from $-\pi\theta/2$ to $\pi/2$.  It is clear from the figure that at very small and large values of $x$ the integrand in (\ref{eq:G(x)_int}) increases very sharply from 0 to the value of 1. In the case $\alpha<1$ the behavior of the integrand will be inverse. In this case, the function is decreasing, and therefore it will sharply decrease from 1 to 0. With a further decrease or increase in the value of $x$ the steep increase (in the case $\alpha>1$) or decrease (in the case $\alpha<1$) in sections will increase. As a result, at some $x$ numerical integration algorithms cannot recognize the monotonic nature of the function and begin to produce an incorrect result.

\begin{figure}
  \centering
  \includegraphics[width=0.48\textwidth]{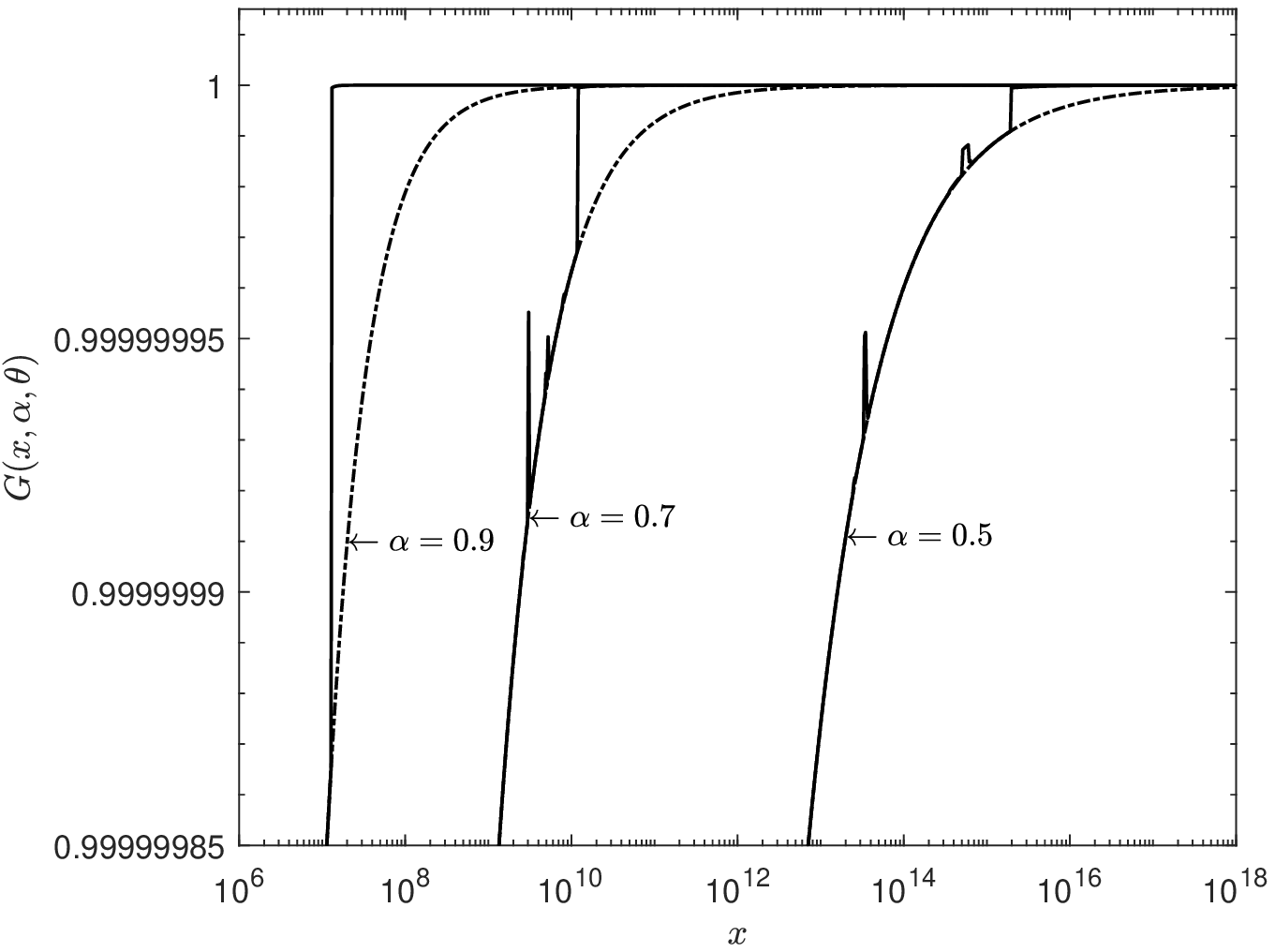}\hfill
  \includegraphics[width=0.48\textwidth]{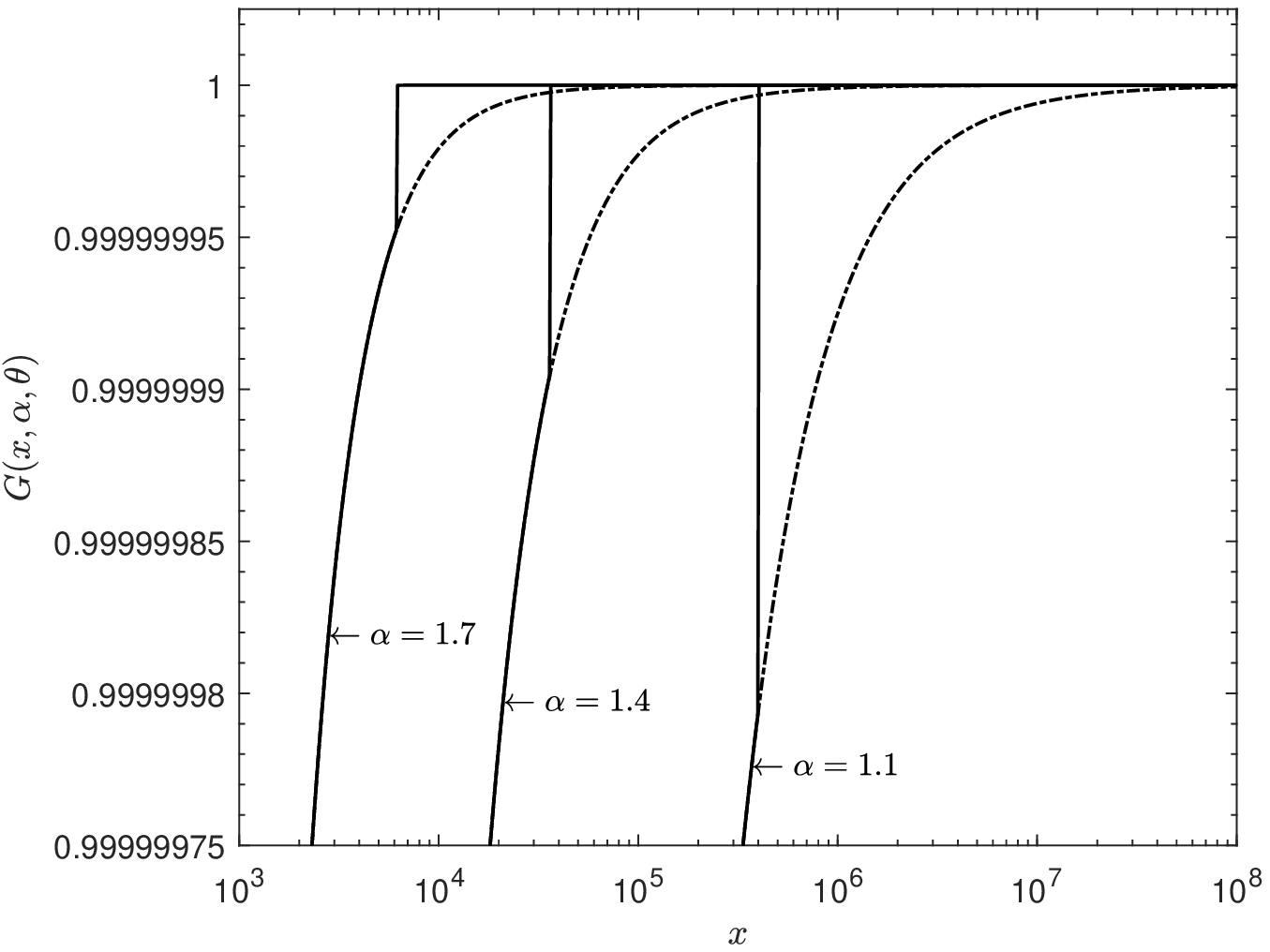}\\
  \caption{ Distribution function $G(x,\alpha,\theta)$ a strictly stable law. The figure on the left is the case $\alpha<1$, the figure on the right is the case $\alpha>1$. The values of the indicator $\alpha$ are given in the figures for all graphs $\theta=0$. The solid curves are integral representation (\ref{eq:G(x)_int}), dash-dotted curves are the representation in a power series  (\ref{eq:cdfCalcExpr})}\label{fig:сdf_x2inf}
\end{figure}

The most suitable method for calculating the distribution function in the case of small and large values of $x$ is to use asymptotic expansions. The problem of calculating the distribution function in the case $x\to0$ was considered in the article \cite{Saenko2022b}. In the case $x\to\infty$ it is expedient to use theorem~\ref{theor:cdfExpansion} and, in particular, the formula (\ref{eq:cdfCalcExpr}). Figure~\ref{fig:сdf_x2inf} shows the results of calculating the distribution function $G(x,\alpha,\theta)$ using the integral representation (\ref{eq:G(x)_int}) (solid curves) and the formula (\ref{eq:cdfCalcExpr}) (dash-dotted curves) at large values of $x$. The left figure shows the case $\alpha<1$, the right one shows the case $\alpha>1$. To calculate the integral in the formula (\ref{eq:G(x)_int}) the Gauss-Kronrod algorithm was used.  One can see from the figures that at large values of $x$ the numerical integration algorithm used is incapable of calculating the integral in (\ref{eq:G(x)_int}) and starts giving an incorrect result. It can also be seen from the figure that the value of the critical coordinate $x_{\mbox{\scriptsize cr}}$, at which the numerical integration algorithm starts calculating the integral incorrectly, depends on the value of $\alpha$. For the value $\alpha=0.5$ the value is $x_{\mbox{\scriptsize cr}}\approx 3.2\cdot10^{13}$,  for the value $\alpha=0.7$ the value is $x_{\mbox{\scriptsize cr}}\approx 3\cdot10^9$,  for the value $\alpha=0.9$ the value is $x_{\mbox{\scriptsize cr}}\approx 10^7$, for the value $\alpha=1.1$ the value is $x_{\mbox{\scriptsize cr}}\approx 4\cdot10^5$, for the value $\alpha=1.4$ the value is $x_{\mbox{\scriptsize cr}}\approx 3.5\cdot10^4$, for the value $\alpha=1.7$ the value is $x_{\mbox{\scriptsize cr}}\approx 6\cdot10^3$.  It is clear that as the value $\alpha$ decreases, the value $x_{\mbox{\scriptsize cr}}$ increases. Thus, at $|x|>x_{\mbox{\scriptsize cr}}$ other methods of calculating the distribution function should be used.

The use of the formula (\ref{eq:cdfCalcExpr}) to calculate the distribution function at $|x|>x_{\mbox{\scriptsize cr}}$ solves the problem completely. It is clear that in the domain $x<x_{\mbox{\scriptsize cr}}$ the calculation results using the integral representation  (\ref{eq:G(x)_int}) and the formula (\ref{eq:cdfCalcExpr}) coincide completely. It should be noted that such a coincidence will be observed in the domain $x_\varepsilon^N\leqslant x\leqslant x_{\mbox{\scriptsize cr}}$.  In the domain $x>x_{\mbox{\scriptsize cr}}$ the numerical integration algorithm no longer makes it possible to obtain the correct value of the distribution function using the integral representation (\ref{eq:G(x)_int}), whereas the use of the formula (\ref{eq:cdfCalcExpr}) does not lead to any calculation difficulties. It should be noted that the number of summands $N=30$ was used to calculate the distribution function with the help of the formula (\ref{eq:cdfCalcExpr}).  The threshold coordinate $x_\varepsilon^N$ for the accuracy level $\varepsilon=10^{-5}$  for the presented graphs has the following values: at $\alpha=0.5, x_\varepsilon^{30}=0.088$, at $\alpha=0.7, x_\varepsilon^{30}=0.402$, at $\alpha=0.9, x_\varepsilon^{30}=1.000$, at $\alpha=1.1, x_\varepsilon^{30}=1.860$, at $\alpha=1.4, x_\varepsilon^{30}=3.552$ and at $\alpha=1.7, x_\varepsilon^{30}=5.612$. As one can see, the values of the threshold coordinate for each of the graphs presented in the figure are significantly less than the range of values of $x$, which is given in the  figures. Consequently, it can be asserted that in the domain $|x|\geqslant x_\varepsilon^N$ to calculate the distribution function, one can use the formula (\ref{eq:cdfCalcExpr}). In this case, the absolute error of calculating the distribution function will not exceed the specified accuracy level $\varepsilon$ and as $x$ increases the absolute error of calculation will only decrease. Thus, the use of theorem~\ref{theor:cdfExpansion} and, in particular, the formula (\ref{eq:cdfCalcExpr}) solves the problem of calculating the distribution function completely at $x\to\infty$.

It should be noted that the presented results are related to standard strict-stable laws, i.e. to laws with scale parameter $\lambda=1$. To transform the distribution function of a standard strictly stable law into a distribution function of a strictly stable law with an arbitrary $\lambda$ one can use remark~7 from the article \cite{Saenko2020b}, (see also\cite{Zolotarev1986, Zolotarev1999}).

\section{Conclusion}

The article considers the problem of calculating the distribution function of a strictly stable law with the characteristic function (\ref{eq:CF_formC}) at large values of the coordinate $x$. The need to solve this problem is dictated by the inability of numerical integration algorithms to calculate the integral correctly in the integral representation (\ref{eq:G(x)_int}) at large $x$. The cause of such difficulties lies in the behavior of the integrand. In this regard, the use of the integral representation to calculate the distribution function at large values of the coordinate is no longer possible, and it is necessary to apply other approaches to solve this problem.

To solve it, it was proposed to use the expansion of the distribution function in a power series at  $x\to\infty$. In the article, such an expansion was obtained, as well as an estimate for the remainder term. The results are formulated in theorem~\ref{theor:cdfExpansion}. The convergence of this series has been studied and it has been shown that in the case of $\alpha<1$ the series is convergent, in the case $\alpha>1$ it is asymptotic, and in the case $\alpha=1$ the series is convergent at $|x|>1$. These results are formulated as a corollary~\ref{corol:cdfConverg}. It should be noted that the results formulated in this corollary for the cases $\alpha<1$ and $\alpha>1$ are not new and generalize the known results related to the convergence of the expansion of the distribution function in a series (see, for example\cite{Zolotarev1986,Uchaikin1999}). Nevertheless, the study of the convergence of the series for the expansion of the distribution function in the case $\alpha=1$ was carried out for the first time. The study of this case showed that the obtained series is convergent at $|x|>1$. In addition, we managed to show that in this case  at $N\to\infty$ this series converges to the distribution function of the generalized Cauchy distribution (\ref{eq:G(x)_a=1}).

The estimate of the remainder term obtained in theorem~\ref{theor:cdfExpansion}, turned out to be very useful in the problem of calculating the distribution function. Using this estimate, we managed to obtain the equation (\ref{eq:x_eps_eq_cdf}) for the threshold coordinate $x_\varepsilon^N$. The threshold coordinate makes it possible to determine the range of coordinates $x$ in which the absolute calculation error does not exceed the required level of accuracy $\varepsilon$ at specified values of $\alpha$ and the number of summands $N$ in the expansion (\ref{eq:G_NInf}). As a result, the formula (\ref{eq:cdfCalcExpr}) is valid for calculating the distribution function.  The calculations performed showed that when using this formula, the absolute error in calculating the distribution function in the domain $|x|>x_\varepsilon^N$ does not exceed the required level of accuracy $\varepsilon$, and in reality is much less than this value. As $|x|$ increases, the absolute calculation error will only decrease. This makes it possible to use the formula (\ref{eq:cdfCalcExpr}) to calculate the distribution function even for those values of $x$, for which the use of the integral representation (\ref{eq:G(x)_int}) turns out to be impossible. Indeed, the calculations have shown that for the integral representation (\ref{eq:G(x)_int}) there is a critical value of the coordinate $x_{\mbox{\scriptsize cr}}$ at which numerical integration algorithms can no longer calculate the integral correctly (see Fig.~\ref{fig:сdf_x2inf}). At the same time, using the formula (\ref{eq:cdfCalcExpr}) does not lead to any calculation difficulties. Thus, using this formula to calculate the distribution function of a strictly stable law in the coordinate range $|x|>x_{\mbox{\scriptsize cr}}$ solves the problem of calculating the distribution function at large $x$.

As noted in the Introduction, the integral representation of the distribution function (\ref{eq:G(x)_int}) has two domains in which numerical methods have difficulties in calculating the integral. These are the ranges of coordinates at $x\to0$ and at $x\to\infty$. This article shows that in the domain of large values of coordinates, the formula (\ref{eq:cdfCalcExpr}) can be used for calculation. The problem of calculating the distribution function at small values of $x$ was solved earlier in the paper  \cite{Saenko2022b}. In this paper, the expansion of the distribution function in a power series at $x\to0$ was obtained and  the area of applicability of this expansion was determined. Thus, if to use the formula (\ref{eq:cdfCalcExpr}) to calculate the distribution function at large values of $x$, at small values of $x$ to use the results of the article \cite{Saenko2022b}, and to use in the intermediate domain the integral representation (\ref{eq:G(x)_int}), then we get an opportunity to calculate correctly the distribution function of a strictly stable law with the characteristic function (\ref{eq:CF_formC})  on the entire real line.

In conclusion, it should be pointed out that there are similar difficulties in calculating the probability density of a strictly stable law.  Calculation difficulties also arise at small and large values of the coordinate $x$. The problem of calculating the probability density is solved in the papers \cite{Saenko2022b,Saenko2023}. These articles show that if in the domain of small coordinates $x$ to     use the probability density expansion from the article \cite{Saenko2022b}, in the domain of large coordinates to use the expansion from the article \cite{Saenko2023}, and in the intermediate domain to use the integral representation for the density probability obtained in the article \cite{Saenko2020b}, then we get an opportunity to calculate the probability density of a strictly stable law correctly on the entire real line. Thus, the problem of calculating the probability density and distribution function of a strictly stable law with the characteristic function (\ref{eq:CF_formC}) on the entire real line turns out to be solved.

\AcknowledgementSection
The author thanks M. Yu. Dudikov for translation of the article into English.

\appendix
\section{Integral representation of the distribution function}\label{sec:IntRepr}
To perform the inverse Fourier transform and obtain the probability density distribution, the following lemma is useful, which determines the inversion formula
\begin{lemma}\label{lem:Inverse}
The probability density distribution $g(x,\alpha,\theta)$ for any admissible set of parameters $(\alpha,\theta)$ and any $x$ can be obtained using the inverse transform formulas
\begin{equation}\label{eq:InverseFormula}
  g(x,\alpha,\theta)=\frac{1}{2\pi}\int_{-\infty}^{\infty}e^{-itx}\hat{g}(t,\alpha,\theta)dt=
  \left\{\begin{array}{c}
           \displaystyle\frac{1}{\pi}\Re\int_{0}^{\infty} e^{itx}\hat{g}(t,\alpha,-\theta)dt,  \\
           \displaystyle\frac{1}{\pi}\Re\int_{0}^{\infty} e^{-itx}\hat{g}(t,\alpha,\theta)dt.
         \end{array}\right.
\end{equation}
\end{lemma}
The proof of this lemma can be found in the article \cite{Saenko2020b}.

There is also an integral representation for the distribution function. For a strictly stable law with a characteristic function (\ref{eq:CF_formC}) it was obtained in the article \cite{Saenko2020b} and formulated as a corollary
\begin{corollary}\label{corol:SSL_cdf}
The distribution function of the stable law $G(x,\alpha,\theta)$ with characteristic function (\ref{eq:CF_formC})  can be represented in the form
\begin{enumerate}
  \item If $\alpha\neq 1$, then for any $|\theta|\leqslant\min(1,2/\alpha-1)$ and $x\neq0$
  \begin{equation}\label{eq:G(x)_int}
    G(x,\alpha,\theta)=\tfrac{1}{2}(1-\sign(x))+\sign(x)G^{(+)}(|x|,\alpha,\theta^*),
    \end{equation}
    where $\theta^*=\theta\sign(x)$,
    \begin{multline}\label{eq:G(x)_x>0}
      G^{(+)}(x,\alpha,\theta)=1-\frac{(1+\theta)}{4}(1+\sign(1-\alpha))\\
      + \frac{\sign(1-\alpha)}{\pi} \int_{-\pi\theta/2}^{\pi/2}\exp\left\{-x^{\alpha/(\alpha-1)} U(\varphi,\alpha,\theta)\right\}d\varphi, \quad x>0,
    \end{multline}
    and $U(\varphi,\alpha,\theta)$ is determined by the expression
    \begin{equation*}        U(\varphi,\alpha,\theta)=\left(\frac{\sin\left(\alpha\left(\varphi+\frac{\pi}{2}\theta\right)\right)}{\cos\varphi}\right)^{\alpha/(1-\alpha)} \frac{\cos\left(\varphi(1-\alpha)-\frac{\pi}{2}\alpha\theta\right)}{\cos\varphi}.
    \end{equation*}
  \item If $\alpha=1$, then for any $-1\leqslant\theta\leqslant1$ and any  $x$
  \begin{equation}\label{eq:G(x)_a=1}
    G(x,1,\theta)=\frac{1}{2}+\frac{1}{\pi}\arctan\left(\frac{x-\sin(\pi\theta/2)}{\cos(\pi\theta/2)}\right).
  \end{equation}
  \item If $x=0$, then for any admissible $\alpha$ and $\theta$
  \begin{equation*}
    G(0,\alpha,\theta)=(1-\theta)/2.
  \end{equation*}
\end{enumerate}
\end{corollary}
\bibliographystyle{elsarticle-num}
\bibliography{d:/bibliography/library}
\end{document}